\documentclass[a4paper]{article}

\usepackage[top=3cm, bottom=3cm, left=3cm, right=3cm]{geometry}
\usepackage[pdftex]{graphicx}
\usepackage{tikz}
\usetikzlibrary{shapes, arrows, arrows.meta, decorations.markings}

\usepackage{amsmath}
\usepackage{amsthm}
\usepackage{amssymb}
\usepackage{mathtools}
\usepackage[integrals]{wasysym}
\usepackage{stmaryrd}

\usepackage{array}

\usepackage{algorithm}
\usepackage{algorithmic}

\usepackage{bbm}
\usepackage[mathcal]{euscript}

\usepackage[utf8]{inputenc}
\usepackage[T1]{fontenc}
\usepackage[english]{babel}
\usepackage{cite}
\usepackage{lmodern}

\usepackage{calc}
\usepackage[inline]{enumitem}
\usepackage[normalem]{ulem}
\usepackage{url}

\usepackage[hidelinks, bookmarks, bookmarksnumbered, pdfstartview={XYZ null null 1.00}]{hyperref}

\theoremstyle{plain}
\newtheorem{theorem}{Theorem}[section]
\newtheorem{lemma}[theorem]{Lemma}
\newtheorem{corollary}[theorem]{Corollary}
\newtheorem{proposition}[theorem]{Proposition}

\theoremstyle{definition}
\newtheorem{definition}[theorem]{Definition}
\newtheorem{remark}[theorem]{Remark}

\DeclareMathOperator{\Span}{Span}

\DeclareMathOperator{\supp}{spt}
\DeclareMathOperator{\dom}{dom}

\DeclareMathOperator*{\argmin}{arg\,min}
\DeclareMathOperator*{\minimize}{minimize}
\DeclareMathOperator{\dist}{d}

\DeclarePairedDelimiter{\norm}{\lVert}{\rVert}
\DeclarePairedDelimiter{\abs}{\lvert}{\rvert}

\DeclarePairedDelimiter{\ceil}{\lceil}{\rceil}

\DeclarePairedDelimiter{\scalprod}{\langle}{\rangle}

\newcommand{\suchthat}{\ifnum\currentgrouptype=16 \mathrel{}\middle|\mathrel{}\else\mid\fi}
\newcommand{\loc}{\mathrm{loc}}
\newcommand{\diff}{\,\mathrm{d}} 

\DeclareMathOperator{\NAG}{NAG}
\newcommand{\NAGall}{\NAG(X,\allowbreak Y,\allowbreak \pi,\allowbreak L,\allowbreak H,\allowbreak m_0)}
\DeclareMathOperator{\Min}{Min}

\DeclareMathOperator{\Opt}{Opt}
\DeclareMathOperator{\OOpt}{\mathbf{Opt}}

\numberwithin{figure}{section}

\makeatletter
\def\case#1{\def\tempa{#1}\futurelet\next\case@i} 
\def\case@i{\ifx\next\bgroup\expandafter\case@ii\else\expandafter\case@end\fi} 
\def\case@ii#1{ 
  \par
  \addvspace{\medskipamount}%
  \noindent\emph{Case~\tempa: #1.}\par\nobreak\smallskip
  \@afterheading%
} 
\def\case@end{ 
  \par
  \addvspace{\medskipamount}%
  \noindent\emph{Case~\tempa.}%
}
\makeatother

\begin{document}

\setlength{\parskip}{1pt plus 1pt minus 1pt} 

\setlist[enumerate, 1]{label={\textnormal{(\alph*)}}, ref={(\alph*)}, leftmargin=0pt, itemindent=*}
\setlist[enumerate, 2]{label={\textnormal{(\roman*)}}, ref={(\roman*)}}
\setlist[description, 1]{leftmargin=0pt, itemindent=*}
\setlist[itemize, 1]{label={\textbullet}, leftmargin=0pt, itemindent=*}

\newlist{hypothesisMFG}{enumerate}{1}
\setlist[hypothesisMFG]{label={\textup{(H\arabic*)}}, ref={(H\arabic*)}, leftmargin=*, widest*=9}
\newlist{hypothesisNAG}{enumerate}{1}
\setlist[hypothesisNAG]{label={\textup{(A\arabic*)}}, ref={(A\arabic*)}, leftmargin=*, widest*=9}


\title{A variational mean field game of controls with free final time and pairwise interactions}

\author{Guilherme Mazanti\thanks{Université Paris-Saclay, CNRS, CentraleSupélec, Inria, Laboratoire des signaux et systèmes, 91190 Gif-sur-Yvette, France.}~\thanks{Fédération de Mathématiques de CentraleSupélec, 91190, Gif-sur-Yvette, France.} \and Laurent Pfeiffer\footnotemark[1]~\footnotemark[2] \and Saeed Sadeghi Arjmand\thanks{Université Paris-Panthéon-Assas, Laboratoire d'économie mathématique et de microéconomie appliquée (LEMMA), place du Panthéon, 75005 Paris, France.}}

\maketitle

\begin{abstract}
This article considers a mean field game model inspired by crowd motion models in which agents aim at reaching a given target set and wish to minimize a cost consisting of an individual running cost, an individual cost depending on the arrival time at the target set, and an interaction running cost, which takes the form of pairwise interactions with other agents through both positions and velocities. We subsume this game under a more general class of games on abstract Polish spaces with pairwise interactions, and prove that the latter games have a variational structure (in the sense that their equilibria can be characterized as critical points of some potential functional) and admit equilibria. We also discuss two a priori distinct notions of equilibria, providing a sufficient condition under which both notions coincide. The results for the games in abstract Polish spaces are applied to our mean field game model, and a numerical illustration concludes the paper.

\bigskip


\end{abstract}

\tableofcontents

\section{Introduction}

Since the introduction of mean field games (MFGs) in 2006 by Jean-Michel Lasry and Pierre-Louis Lions \cite{Lasry2006JeuxI, Lasry2006JeuxII, Lasry2007Mean} and the simultaneous independent works by Peter E.~Caines, Minyi Huang, and Roland P.~Malhamé \cite{Huang2006Large, Huang2007Large, Huang2003Individual}, many authors have proposed MFG models for (or inspired by) the analysis of crowd motion \cite{Lachapelle2011Mean, Burger2013Mean, Benamou2017Variational, Bagagiolo2022Optimal, Cristiani2023Generalized, Mazanti2019Minimal, Santambrogio2021Cucker}. While the mathematical description of the movement of crowds of people has attracted the interest of several researchers along the years (see, e.g., \cite{Aylaj2022Crowd, Helbing1995Social, Henderson1971Statistics, Maury2019Crowds, Muntean2014Collective, Hughes2002Continuum}), the MFG approach to crowd motion has the particularity that, by considering agents of the crowd to be rational, one may consider that each agent chooses their trajectory based not only on the present or past states of other agents, but also on a rational anticipation of their future behavior. Our aim in this paper is to propose and study an MFG model inspired by crowd motion in which agents interact through a Cucker--Smale-type interaction, in the spirit of \cite{Santambrogio2021Cucker}.

To the authors' knowledge, the first work to be fully dedicated to a mean field game model for crowd motion is \cite{Lachapelle2011Mean}, which proposes an MFG model for a two-population crowd with trajectories perturbed by additive Brownian motion and considers both their stationary distributions and their evolution on a prescribed time interval. Other works have later proposed MFG models for crowd motion taking into account different characteristics. To cite a few examples, let us mention \cite{Burger2013Mean}, which considers the fast exit of a crowd and proposes a mean field game model which is studied numerically; \cite{Cardaliaguet2016First}, which is not originally motivated by the modeling of crowd motion but considers an MFG model with a density constraint, which is a natural assumption in some crowd motion models (see also \cite{Meszaros2015Variational} for the case of second-order mean field games with density constraints); \cite{Achdou2017Mean, Carlini2018Fully}, which study a two-population MFG model motivated by urban settlements; \cite{Benamou2017Variational}, which presents numerical simulations for some variational mean field games related to crowd motion; \cite{Bagagiolo2022Optimal}, which, inspired by cities crowded by tourists such as Venice, Italy, proposes a mean field game model for the movement of pedestrian tourists; \cite{Cristiani2023Generalized}, which provides a generalized MFG model for pedestrians with limited predictive abilities; or also \cite{Mazanti2019Minimal, Dweik2020Sharp, Ducasse2022Second, Sadeghi2022Multi, Sadeghi2022Nonsmooth, Mazanti2024Note}, which consider MFG models in which the goal of each agent is to reach a given target set in minimal time, interactions between agents being modeled through a congestion-dependent maximal velocity. Other works studying mean field games motivated by or related to crowd motion include \cite{Achdou2019Mean, Bardi2021Convergence, Aurell2020Behavior, Djehiche2017Mean, Carlini2017Semi}.

A common feature of several MFG models for crowd motion is that, at a given time $t$, an agent at a space position $x$ will interact with other agents only through the distribution $m_t$ of their positions in the state space at time $t$. While this already provides models with interesting features, a position-only interaction fails to capture important aspects of crowd motion. Indeed, given two agents at a certain distance from one another, it is natural to expect that their interaction will be qualitatively different depending on whether they are moving towards or away from one another, since in the first case they are expected to deviate in order to avoid a collision, while, in the second case, their movements are not expected to have a mutual interference. This motivates the study of MFG models for crowd motion in which the interaction through agents is based not only on their positions, but also on their velocities, or, equivalently for many models, on their controls.

Mean field game models in which the dependence on other agents takes into account not only their positions but also their controls are called \emph{MFGs of controls}. Their analysis, started in \cite{Gomes2014Existence, Gomes2016Extended}, has been much developed in recent years (see, e.g., \cite{Cardaliaguet2018Mean, Djete2023Large, Camilli2023Quasi, Bonnans2023Lagrangian, Graber2023Master, Kobeissi2022Mean, Kobeissi2022Classical, Santambrogio2021Cucker, Graber2021Weak, Graber2025Remarks, Hofer2026Optimal}, and, in particular, the PhD thesis \cite{Kobeissi2020Contributions}). In the context of crowd motion, \cite{Santambrogio2021Cucker} has proposed a model\footnote{We point out that the authors of \cite{Santambrogio2021Cucker} argue that whether the model they analyze is indeed an MFG of controls is up to debate, and that ``classifying it or not as an MFG of controls is a matter of taste''.} which assumes that agents evolve in a given set $\Omega \subset \mathbb R^d$, that the trajectory $\gamma$ of an agent choosing a control $u$ is described by the control system $\dot\gamma(t) = u(t)$, and that an agent chooses their control $u$ in order to minimize a cost of the form
\begin{equation}
\label{eq:intro-cost-CS}
\int_{0}^{T} \left(\frac{\delta}{2}\abs{\dot \gamma(t)}^2 + \frac{\lambda}{2} \int_{\mathcal C([0, T], \Omega)} \eta\left(\gamma (t) - \widetilde \gamma (t)\right)\abs*{\dot \gamma(t) - \dot{\widetilde \gamma}(t)}^2 \diff Q(\widetilde \gamma)\right) \diff t + \Psi(\gamma(T)),
\end{equation}
where $\delta$ and $\lambda$ are positive constants, $\eta\colon \mathbb R^d \to \mathbb R_+$ is an interaction kernel, $T > 0$ is a fixed time horizon, and $Q$ is a probability measure on the set of all continuous trajectories $\mathcal C([0, T], \Omega)$ describing the distribution of trajectories of all agents.

Given an initial distribution of agents $m_0$, described mathematically as a probability measure on $\Omega$, an equilibrium of this MFG model with initial condition $m_0$ is then a probability measure $Q$ on $\mathcal C([0, T], \Omega)$ whose evaluation at time $0$ is\footnote{That is, ${e_0}_{\#} Q = m_0$, using the notation described in page~\pageref{notation} below.} $m_0$ and which is concentrated on trajectories $\gamma$ that are optimal for \eqref{eq:intro-cost-CS} with a fixed initial condition. It is proved in \cite[Proposition~3.1]{Santambrogio2021Cucker} that this MFG model is \emph{variational} (also called \emph{potential}), i.e., that there exists a function $\mathcal J$, defined over the space of probability measures on $\mathcal C([0, T], \Omega)$ whose evaluation at time $0$ is $m_0$, such that (local) minimizers of $\mathcal J$ are equilibria of the MFG with initial condition $m_0$. In particular, this variational structure allows one to prove the existence of an equilibrium for the MFG by proving the existence of a (global) minimizer of $\mathcal J$, which is done in \cite[Theorem~3.1]{Santambrogio2021Cucker}. Additional properties of equilibria are also provided in \cite{Santambrogio2021Cucker}, such as the fact that, under suitable additional assumptions, optimal trajectories are $\mathcal C^{1, 1}$ and unique. We also refer to \cite{Hofer2026Optimal} for a related second-order potential MFG of controls with application in particular to a flocking model in the spirit of \eqref{eq:intro-cost-CS}.

Note that this way of describing equilibria as probability measures on the space of trajectories corresponds to the \emph{Lagrangian formulation} of mean field games. The Lagrangian formulation is a classical approach in optimal transport problems (see, e.g., \cite{Santambrogio2015Optimal, Ambrosio2005Gradient, Villani2009Optimal}), which has been used for instance in \cite{Brenier1989Least} to study incompressible flows, in \cite{Carlier2008Optimal} for Wardrop equilibria in traffic flow, or in \cite{Bernot2009Optimal} for branched transport problems. The use of the Lagrangian approach in mean field games dates back at least to \cite{Cardaliaguet2016First, Cardaliaguet2015Weak}, and since then it has been used in several works, such as \cite{Benamou2017Variational, Cannarsa2018Existence, Cannarsa2019C11, Cannarsa2021Mean, Fischer2021Asymptotic, Graber2025Remarks, Mazanti2019Minimal, Dweik2020Sharp, Sadeghi2022Multi, Sadeghi2022Nonsmooth, Sadeghi2021Characterization, Mazanti2024Note}. In particular, a discussion on the Lagrangian approach for potential MFGs of controls is provided in \cite[Section~2.5]{Graber2025Remarks}.

The model of \cite{Santambrogio2021Cucker} is inspired by the Cucker--Smale model. Introduced in \cite{Cucker2007Emergent} to model the evolution of a flock of birds, the Cucker--Smale model assumes that a finite number of agents (which represent birds in the original model) evolve according to a prescribed law taking into account the fact that agents wish to align their velocities with others, and that this alignment effect is stronger when the birds are closer. In \eqref{eq:intro-cost-CS}, this is modeled by the term $\eta\left(\gamma (t) - \widetilde \gamma (t)\right)\abs*{\dot \gamma(t) - \dot{\widetilde \gamma}(t)}^2$, which penalizes large differences in velocities for agents that are close if one assumes that $\eta$ is a Gaussian-like kernel (for instance, nonnegative, depending only on the modulus of its argument, and nonincreasing with it). Note also that the interactions between agents in \eqref{eq:intro-cost-CS} are pairwise: for a given agent with a trajectory $\gamma$, its total interaction is the integral of the interactions of the pairs of trajectories $(\gamma, \widetilde\gamma)$ for all trajectories $\widetilde\gamma$ distributed according to $Q$.

Most mean field games considered in the literature, including the one from \cite{Santambrogio2021Cucker} but also those from \cite{Lasry2006JeuxII, Lasry2007Mean, Huang2006Large, Bagagiolo2022Optimal, Benamou2017Variational, Lachapelle2011Mean, Burger2013Mean, Cardaliaguet2016First, Achdou2017Mean, Carlini2018Fully, Achdou2019Mean, Aurell2020Behavior, Carlini2017Semi, Djehiche2017Mean}, assume that the time interval $[0, T]$ of the game is fixed, with all agents starting at time $0$ and ending their movement at time $T$. However, in many practical applications, including crowd motion but also in economics, for instance, it is interesting to study models in which the final time for the movement of an agent is free, and is actually part of their minimization criterion. In particular, agents may leave the game before its end. MFGs with a free final time have been considered, for instance, in \cite{Carmona2017Mean, Burzoni2023Mean}, which consider mean field games with applications to bank run, i.e., situations in which clients of a bank, believing that the bank is about to fail, withdraw all their money, and try to choose the time to withdraw the money in an optimal way. These models belong to a more general class of mean field game problems known as mean field games of optimal stopping, in which the main choice of an agent is when to stop the game \cite{Bertucci2018Optimal, Gomes2015Obstacle, Nutz2018Mean, Bouveret2020Mean, Nutz2020Convergence, Bertucci2021Monotone}. Other works, such as \cite{Chan2017Fracking, Graber2020Mean}, consider MFGs with free final time for the production of exhaustible resources, in which firms, who wish to maximize their profit, produce goods based on exhaustible resources, and they leave the game when they deplete their capacities. These games can be seen as mean field games with an absorbing boundary, i.e., in which agents who reach a certain part of the boundary of the domain immediately leave the game \cite{Campi2018NPlayer, DiPersio2022Master}. In the context of crowd motion, MFG models with free final time were developed in \cite{Cristiani2023Generalized, Ducasse2022Second, Dweik2020Sharp, Mazanti2024Note, Mazanti2019Minimal, Sadeghi2022Multi, Sadeghi2022Nonsmooth}.

In this paper, we generalize the model from \cite{Santambrogio2021Cucker} by considering that \begin{enumerate*}[label={(\roman*)}]
\item\label{item:novel-1} one has a more general individual running cost than $\frac{\delta}{2}\abs{\dot \gamma(t)}^2$ from \eqref{eq:intro-cost-CS};
\item\label{item:novel-2} one has a more general interaction cost than $\eta\left(\gamma (t) - \widetilde \gamma (t)\right)\abs*{\dot \gamma(t) - \dot{\widetilde \gamma}(t)}^2$ from \eqref{eq:intro-cost-CS};
\item\label{item:novel-3} the final time of the movement of an agent is not a prescribed time $T$, but is part of the optimization criterion of each agent; and
\item\label{item:novel-4} the aim of each agent is to reach a given common target set $\Xi$.
\end{enumerate*}
More precisely, we shall consider that each agent of the game minimizes a cost of the form
\begin{equation}
\label{eq:intro-cost-general}
\int_{0}^{+\infty} \ell(t, \gamma(t), \dot{\gamma}(t))\diff t + \Psi(\tau(\gamma)) + \int_{\Gamma} \int_{0}^{\tau(\gamma)\wedge \tau(\widetilde \gamma)} h(t, \gamma(t), \widetilde{\gamma}(t), \dot{\gamma}(t), \dot{\widetilde{\gamma}}(t)) \diff t\diff Q(\widetilde{\gamma}),
\end{equation}
where $\ell$, $h$, and $\Psi$ are given functions and, for a given trajectory $\gamma$, $\tau(\gamma)$ denotes the exit time of $\gamma$, i.e., the first time at which $\gamma$ reaches the target set $\Xi$ (or $\tau(\gamma) = +\infty$ if $\gamma$ does not reach $\Xi$). While \ref{item:novel-1} and \ref{item:novel-2} do not bring many technical difficulties with respect to \cite{Santambrogio2021Cucker}, novel ideas are required to treat the free final-time setting from \ref{item:novel-3} and \ref{item:novel-4} and address the difficulties coming from a noncompact time interval and the lack of continuity of the exit-time function $\tau(\cdot)$.

It turns out that the potential structure highlighted in \cite{Santambrogio2021Cucker} is not specific to the MFG considered in that reference nor to the MFG described by \eqref{eq:intro-cost-general} and comes instead mostly from the pairwise nature of the interaction between agents. Hence, instead of working directly with \eqref{eq:intro-cost-general}, we consider in most of the paper a more general non-atomic game model in abstract Polish spaces (which are not necessarily spaces of trajectories) in which agents minimize the sum of an individual cost and a pairwise interaction cost. The analysis of more general non-atomic games and their potential structure has also been the subject of other works in the literature, such as \cite{Blanchet2016Optimal}, which considers a non-atomic game which may contain pairwise interaction terms as a particular case. We point out, however, that the model we consider cannot be subsumed under that of \cite{Blanchet2016Optimal} (see Remark~\ref{remk:Blanchet-Carlier}).

In addition to showing, as in \cite{Santambrogio2021Cucker}, that the non-atomic game we consider has a potential structure in the sense that local minimizers of a potential $\mathcal J$ are equilibria of the game, we go further in the potential formulation by showing (Theorem~\ref{thm:potential}) that the set of equilibria of the game is exactly the set of critical points of its potential function $\mathcal J$. We then deduce (Theorem~\ref{thm:exist-abstract}) existence of an equilibrium by proving the existence of a minimizer for $\mathcal J$. We also discuss two a priori distinct notions of equilibria, strong and weak equilibria, and prove (Theorem~\ref{thm:strong-iff-equilibrium}) that, under suitable assumptions, both notions coincide for our model. Finally, we apply the results obtained for the general non-atomic game to the MFG consisting in each agent minimizing \eqref{eq:intro-cost-general}, proving existence of equilibria (Theorem~\ref{thm:3:exist}) and equivalence between strong and weak equilibria (Theorem~\ref{thm:3:equiv}).

The question of uniqueness of equilibria is not addressed in the present work. In mean field games, uniqueness is often related to suitable monotonicity assumptions on the interaction terms, as in the classical references \cite{Lasry2006JeuxI, Lasry2006JeuxII, Lasry2007Mean}, or displacement monotonicity assumptions, as in \cite{Meszaros2024Mean, Gangbo2022Mean, Ahuja2016Wellposedness}. It can also be obtained under smallness assumptions on the horizon of the game or on the interaction between agents \cite{Lasry2007Mean, Carmona2013Probabilistic, Carmona2013Mean}, and recent works such as \cite{Graber2023Monotonicity, Mou2025Mean} have also proposed other uniqueness conditions. In variational settings, uniqueness may also follow from strict convexity properties of the potential functional (see, e.g., \cite{Benamou2017Variational}). It is not immediate to apply or adapt these standard techniques to our setting, in particular since the potential functional used in this paper is not convex in general (see Remark~\ref{remk:not-convex}), and therefore providing conditions ensuring uniqueness of equilibria for our model remains an open question. We also refer to \cite[Remark~4.9]{Dweik2020Sharp}, \cite[Remark~7.1]{Mazanti2019Minimal}, and \cite[Remark~5.7]{Sadeghi2022Multi} for discussions on (non-)uniqueness of equilibria for related mean field game models motivated by crowd motion but without control interaction.
 
The sequel of the paper is organized as follows. Section~\ref{Cucker-Smale:The Model} presents the MFG model consisting on each agent minimizing \eqref{eq:intro-cost-general}, providing our standing assumptions on the functions $\ell$, $h$, and $\Psi$ from \eqref{eq:intro-cost-general} as well as the precise definition of the exit time function $\tau(\cdot)$ and the definition of equilibrium of the game. Section~\ref{sec:abstract} moves to the analysis of a more general non-atomic game model in Polish spaces, presenting the model and its assumptions, providing its elementary properties in Section~\ref{sec:first-properties}, proving its potential structure in Section~\ref{sec:potential}, deducing existence of equilibria in Section~\ref{sec:existence-of-equilibria}, and discussing the equivalence of the a priori distinct notions of strong and weak equilibria in Section~\ref{sec:strong-equilibria}. These general results are applied to the MFG model from Section~\ref{Cucker-Smale:The Model} in Section~\ref{sec:application}. The paper is concluded by a numerical simulation in Section~\ref{sec:illustration} illustrating behavior of the equilibria in an $N$-player setting.

\medskip

\noindent\textbf{Notation.}\label{notation} In this paper, $d$ is a fixed positive integer, the set of nonnegative real numbers is denoted by $\mathbb R_+$, the set of nonnegative rational numbers is denoted by $\mathbb Q_+$, and we define $\overline{\mathbb R}_+$ as $\overline{\mathbb R}_+ = \mathbb R_+ \cup \{+\infty\}$. For a given $\Omega \subset \mathbb R^d$ nonempty, open, and bounded, we let $\mathcal{C}(\mathbb R_+;\overline \Omega)$ be the spaces of continuous curves from $\mathbb R_+$ to $\overline \Omega$ endowed with the topology of uniform convergence on compact time intervals, with which it is a Polish space. For simplicity, we denote $\mathcal{C}(\mathbb R_+; \overline \Omega)$ by $\Gamma$.

Given a Polish space $X$, we denote the space of Borel probability measures on $X$ by $\mathcal{P}(X)$ and endow it with the topology of weak convergence of measures. The support of $\mu \in \mathcal{P}(X)$ is denoted by $\supp(\mu)$, and is defined as the set of all points $x \in X$ such that $\mu(N_x) > 0$ for every neighborhood $N_x$ of $x$. We denote the Dirac measure centered at a point $x_0 \in X$ by $\delta_{x_0}$.

For two metric spaces $X$ and $Y$ endowed with their Borel $\sigma$-algebras and a Borel-measurable map $f\colon X \to Y$, the pushforward of a measure $\mu$ on $X$ through $f$ is the measure $f_{\#} \mu$ on $Y$ defined by $f_{\#} \mu (B) = \mu(f^{-1}(B))$ for every Borel subset $B$ of $Y$. For $t\in \mathbbm R_+$, we denote by $e_t\colon \Gamma \to \overline \Omega$ the evaluation map at time $t$, defined by $e_t(\gamma) = \gamma(t)$ for every $\gamma \in \Gamma$.

\section{Mean field game model with pairwise interactions}
\label{Cucker-Smale:The Model}

We consider here that agents move in $\overline\Omega$, where $\Omega$ is a nonempty open bounded subset of $\mathbb R^d$, and that their goal is to reach a certain target $\Xi \subset \overline\Omega$, which is a nonempty closed set. The initial distribution of agents $m_0 \in \mathcal P(\overline\Omega)$ at time $t = 0$ is known, agents are submitted to the trivial control system $\dot\gamma(t) = u(t)$, and the goal of an agent starting at position $x_0$ is to choose a control $u\colon \mathbb R_+ \to \mathbb R^d$ such that $\gamma(t) \in \overline\Omega$ for every $t \in \mathbb R_+$ and the cost $F(\gamma, Q)$ defined by
\begin{multline}
\label{eq:cost}
F(\gamma, Q) = \int_{0}^{+\infty} \ell(t, \gamma(t), \dot{\gamma}(t))\diff t + \Psi(\tau(\gamma)) + \int_{\Gamma} \int_{0}^{\tau(\gamma)\wedge \tau(\widetilde \gamma)} h(t, \gamma(t), \widetilde{\gamma}(t), \dot{\gamma}(t), \dot{\widetilde{\gamma}}(t)) \diff t\diff Q(\widetilde{\gamma})
\end{multline}
is minimized among all trajectories $\gamma$ starting at the same position $x_0$. In \eqref{eq:cost}, $\ell\colon \mathbb R_+ \times \overline\Omega \times \mathbb R^d \to \mathbb R_+$, $h\colon \mathbb R_+ \times \overline\Omega \times \overline\Omega \times \mathbb R^d \times \mathbb R^d \to \mathbb R_+$, and $\Psi\colon \mathbb R_+ \to \mathbb R_+$ are given functions, $Q \in \mathcal P(\Gamma)$ represents the distribution of the trajectories of all agents, and $\tau\colon \Gamma \to \overline{\mathbb R}_+$ is the first exit time function, defined, for $\gamma \in \Gamma$, by
\begin{equation}
\label{eq:defi3-tau}
\tau(\gamma) = \inf\{t \in \mathbb R_+ \suchthat \gamma(t) \in \Xi\},
\end{equation}
with the convention $\inf\emptyset = +\infty$. We also set, by convention, $F(\gamma, Q) = +\infty$ if $\gamma$ is not absolutely continuous, if $Q$ gives positive measure to a set of functions which are not absolutely continuous, or if $\widetilde\gamma \mapsto \int_{0}^{\tau(\gamma)\wedge \tau(\widetilde \gamma)} h(t, \gamma(t), \widetilde{\gamma}(t), \dot{\gamma}(t), \dot{\widetilde{\gamma}}(t)) \diff t$ is not $Q$-integrable.

Let us provide the definition of equilibrium of this mean field game.

\begin{definition}
\label{def:equilibrium-3}
Let $m_0 \in \mathcal P(\overline\Omega)$. A measure $Q \in \mathcal P(\Gamma)$ is said to be an \emph{equilibrium} with initial condition $m_0$ if ${e_0}_{\#} Q = m_0$, $\int_{\Gamma} F(\gamma, Q) \diff Q(\gamma) < +\infty$, and $Q$-almost every $\gamma$ satisfies
\[
F(\gamma, Q) = \inf_{\substack{\omega \in \Gamma \\ \omega(0) = \gamma(0)}} F(\omega, Q).
\]
\end{definition}

Let us now introduce the main assumptions on $\Omega$, $\Xi$, $\ell$, $h$, and $\Psi$ that we will need for the analysis of this mean field game.

\begin{hypothesisMFG}
\item\label{Hypo3-Omega} $\Omega \subset \mathbb R^d$ is nonempty, open, and bounded.

\item\label{Hypo3-Gamma} $\Xi \subset \overline\Omega$ is nonempty and closed.

\item \label{l is superlinear} The function $\ell\colon\mathbb R_+\times \overline{\Omega}\times \mathbb R^d \to \mathbb R_+$ is such that
\begin{enumerate}
\item $t \mapsto \ell(t, x, p)$ is measurable for every $(x, p) \in \overline\Omega \times \mathbb R^d$;
\item $(x, p) \mapsto \ell(t, x, p)$ is continuous for almost every $t \in \mathbb R_+$;
\item $p \mapsto \ell(t, x, p)$ is convex for almost every $t \in \mathbb R_+$ and every $x \in \overline\Omega$; and
\item\label{item:lower-bound-ell} there exist constants $\alpha > 0$ and $\theta > 1$ such that
\[
\ell(t,x, p) \geq \alpha \abs*{p}^{\theta}, \qquad  \text{ for all } (t, x, p) \in \mathbb R_+ \times \overline{\Omega} \times \mathbb R^d.
\]
\end{enumerate}

\item \label{h convex and symmetric} The function $h\colon \mathbb R_+ \times \overline{\Omega}\times \overline{\Omega}\times  \mathbb R^d\times \mathbb R^d \to \mathbb R_+$ is such that
\begin{enumerate}
\item $t \mapsto h(t, x, \widetilde x, p, \widetilde p)$ is measurable for every $(x, \widetilde x, p, \widetilde p) \in \overline\Omega \times \overline\Omega \times \mathbb R^d \times \mathbb R^d$;
\item $(x, \widetilde x, p, \widetilde p) \mapsto h(t, x, \widetilde x, p, \widetilde p)$ is continuous for almost every $t \in \mathbb R_+$;
\item $(p, \widetilde p) \mapsto h(t, x, \widetilde x, p, \widetilde p)$ is convex for almost every $t \in \mathbb R_+$ and every $(x, \widetilde x) \in \overline\Omega \times \overline\Omega$;
\item\label{item:H-symmetric} for all $(t, x, \widetilde x, p, \widetilde p) \in \mathbb R_+ \times \overline{\Omega} \times \overline{\Omega} \times \mathbb R^d \times \mathbb R^d$, we have the symmetry property
\[
h(t, x, \widetilde{x}, p, \widetilde{p}) = h(t, \widetilde{x}, x,  \widetilde{p}, p); \text{ and}
\]
\item there exist constants $C > 0$ and $\beta \in (0, \theta]$, where $\theta$ is the constant from \ref{l is superlinear}\ref{item:lower-bound-ell}, such that 
\[
h(t, x, \widetilde{x}, p, \widetilde{p}) \leq C(\abs*{p}^{\beta} + \abs*{\widetilde{p}}^{\beta}), \qquad  \text{ for all } (t, x, \widetilde x, p, \widetilde p) \in \mathbb R_+ \times \overline{\Omega} \times \overline{\Omega} \times \mathbb R^d \times \mathbb R^d.
\]
\end{enumerate}

\item\label{Psi1} The function $\Psi\colon \mathbb R_+ \to \mathbb R_+$ is lower semicontinuous, nondecreasing, and there exist positive constants $a$ and $b$ such that
\[
\Psi(t) \ge at - b, \qquad \text{ for all } t\in \mathbb R_+.
\]

\item \label{individualCost} There exists a constant $\kappa > 0$ such that, for every $x_0 \in \overline\Omega$, there exists an absolutely continuous curve $\gamma\in \Gamma$ with $\gamma(0)=x_0$ and
\[
\int_0^{+\infty} \ell(t, \gamma(t), \dot\gamma(t)) \diff t + \Psi(\tau(\gamma)) \leq \kappa.
\]
\end{hypothesisMFG}

Note that \ref{individualCost} can be seen as an accessibility assumption: wherever in $\overline\Omega$ is the starting point $x_0$, there always exists an absolutely continuous trajectory allowing an agent to reach the target set $\Xi$ within a time that is uniformly bounded, and with a running cost that is also uniformly bounded. It can be ensured, for instance, if $\Omega$ is connected and has smooth boundary.

In the sequel, instead of directly studying this mean field game, we will consider in Section~\ref{sec:abstract} a more general non-atomic game containing this MFG as a particular case. We will prove that this non-atomic game has a potential structure and use this fact to deduce the existence of equilibria. Finally, we will apply the results obtained for the general non-atomic game to the above MFG in Section~\ref{sec:application}.

\section{A general non-atomic game with pairwise interactions}
\label{sec:abstract}

As a preliminary step towards the study of the mean field game model we are interested in this paper, we consider here a more general non-atomic game with pairwise interactions. Let $X$ and $Y$ be Polish spaces and $\pi\colon X \to Y$, $L\colon X \to \overline{\mathbb R}_+$, and $H\colon X \times X \to \overline{\mathbb R}_+$ be Borel-measurable functions. Each agent of the game is assumed to be associated with an element $y \in Y$ (but this association is not necessarily injective: different agents can be associated with the same element of $Y$), and we assume that the distribution of agents according to the elements associated with them in $Y$ is described by a known probability measure $m_0 \in \mathcal P(Y)$. An agent of the game associated with some $y \in Y$ wishes to minimize the cost $F(x, Q)$ given by
\begin{equation}
\label{eq:abstract-cost}
F(x, Q) = L(x) + \int_X H(x, \widetilde x) \diff Q(\widetilde x)
\end{equation}
with the constraint $\pi(x) = y$, i.e., the agent wants to choose $x \in X$ solving the minimization problem
\begin{equation}
\label{eq:minimization-abstract}
\Min(y, Q)\colon
\left\{
\begin{aligned}
\minimize_{x \in X } {} & F(x, Q) \\
\text{ subject to } & \pi(x) = y,
\end{aligned}
\right.
\end{equation}
where $Q \in \mathcal P(X)$ denotes the distribution of the choices of all agents. Note that the cost $F(x, Q)$ can be interpreted as containing an individual cost $L(x)$ and an average pairwise interaction cost $\int_X H(x, \widetilde x) \diff Q(\widetilde x)$. In the sequel, we refer to this non-atomic game as $\NAGall$ and we look for equilibria of this game, according to the following definition.

\begin{definition}
\label{def:equilibrium-potential}
Let $X$ and $Y$ be Polish spaces, $\pi\colon X \to Y$, $L\colon X \to \overline{\mathbb R}_+$, and $H\colon X \times X \to \overline{\mathbb R}_+$ be Borel-measurable functions, and $m_0 \in \mathcal P(Y)$. We say that $Q \in \mathcal P(X)$ is an \emph{equilibrium} of $\NAGall$ if $\pi_{\#} Q = m_0$, $\int_X F(x, Q) \diff Q(x) < +\infty$, and $Q$-almost every $x \in X$ solves $\Min(\pi(x), Q)$ with $F$ given by \eqref{eq:abstract-cost}.
\end{definition}

\begin{remark}
It is usual, in the literature on non-atomic games, to assume that the continuum set of players is \([0, 1]\) endowed with the Lebesgue measure over the Lebesgue or Borel \(\sigma\)-algebra (see, e.g., \cite{Aumann1974Values, Schmeidler1973Equilibrium}). In our framework, we are not interested in the players individually, but only in the distribution of their \emph{types} in \(Y\). For this reason, we make no assumption on the measure \(m_0\), and in particular this measure may have atoms. Our framework can be seen as a non-atomic game by regarding \(m_0\) as the pushforward of the Lebesgue measure in \([0, 1]\) through a measurable map associating, with each player in \([0, 1]\), its type in \(Y\). Atoms in \(m_0\) may thus come from the non-injectivity of such a map and represent the fact that a non-negligible amount of players share the same type.
\end{remark}

\begin{remark}
\label{remk:concrete}
The mean field game from Section~\ref{Cucker-Smale:The Model} is a particular case of $\NAGall$. To see that, it suffices to set $Y = \overline\Omega$, $X = \Gamma$ with the topology of uniform convergence on compact sets, $\pi = e_0$, the evaluation map at time zero, and define the functions $L\colon \Gamma \to \overline{\mathbb R}_+$ and $H\colon \Gamma \times \Gamma \to \overline{\mathbb R}_+$ by
\[
L(\gamma) = \begin{dcases*}
\int_0^{+\infty} \ell(t, \gamma(t), \dot\gamma(t)) \diff t + \Psi(\tau(\gamma)) & if $\gamma$ is absolutely continuous, \\
+\infty & otherwise
\end{dcases*}
\]
and
\[
H(\gamma, \widetilde\gamma) = \begin{dcases*}
\int_0^{\tau(\gamma) \wedge \tau(\widetilde\gamma)} h(t, \gamma(t), \widetilde\gamma(t), \dot\gamma(t), \dot{\widetilde\gamma}(t)) \diff t & if $L(\gamma) < +\infty$ and $L(\widetilde\gamma) < +\infty$, \\
+\infty & otherwise.
\end{dcases*}
\]
We will justify later (see Proposition~\ref{prop:same-equilibria}) that, with these definitions, the notions of equilibria from Definitions~\ref{def:equilibrium-3} and \ref{def:equilibrium-potential} coincide.
\end{remark}

\begin{remark}
\label{remk:Blanchet-Carlier}
For an individual agent associated with a given element $y \in Y$, their cost \eqref{eq:abstract-cost} and their minimization problem \eqref{eq:minimization-abstract} are particular cases of the more general problem of, given $y \in Y$, 
\begin{equation}
\label{eq:Blanchet-Carlier}
\minimize_{x \in X} c(x, y) + \mathcal V[Q](x),
\end{equation}
without the constraint $\pi(x) = y$. Indeed, \eqref{eq:Blanchet-Carlier} reduces to \eqref{eq:minimization-abstract} with the choice $c(x, y) = L(x) + \chi_{\operatorname{graph}(\pi)}(x, y)$ and $\mathcal V[Q](x) = \int_X H(x, \widetilde x) \diff Q(\widetilde x)$, where $\chi_{\operatorname{graph}(\pi)}$ is the characteristic function of the graph of $\pi$, i.e., $\chi_{\operatorname{graph}(\pi)}(x, y) = 0$ if $y = \pi(x)$ and $\chi_{\operatorname{graph}(\pi)}(x, y) = +\infty$ otherwise.

The non-atomic game in which each agent minimizes \eqref{eq:Blanchet-Carlier} was previously considered in the literature, for instance, in \cite{Blanchet2016Optimal}, where the authors study equilibria of the game, named in that context as \emph{Cournot--Nash equilibria}, through a variational approach, exploiting in particular links between such equilibria and optimal transport problems. These ideas, however, cannot be directly applied to the non-atomic game $\NAGall$ that we consider in this article, as some key assumptions of \cite{Blanchet2016Optimal} cannot be satisfied in our context. In particular, \cite{Blanchet2016Optimal} requires $c$ to be continuous in order to obtain good properties for the associated Wasserstein distance $W_c$, but such a continuity assumption can never be satisfied in our model due to the presence of the characteristic function $\chi_{\operatorname{graph}(\pi)}$ in the expression of $c$.
\end{remark}

Let us now state the main assumptions under which we will consider the abstract non-atomic game $\NAGall$.

\begin{hypothesisNAG}
\item\label{Hypo-XY-Polish} $X$ and $Y$ are Polish spaces.
\item\label{Hypo-pi} $\pi\colon X \to Y$ is continuous.
\item\label{Hypo-LH} $L\colon X \to \overline{\mathbb R}_+$ is lower semicontinuous and $H\colon X \times X \to \overline{\mathbb R}_+$ is Borel measurable.
\item\label{Hypo-H-symmetric} $H$ is symmetric, i.e., $H(x, \widetilde x) = H(\widetilde x, x)$ for every $(x, \widetilde x) \in X \times X$. 
\item\label{Hypo-J} The function $X \times X \ni (x, \widetilde x) \mapsto L(x) + L(\widetilde x) + H(x, \widetilde x) \in \overline{\mathbb R}_+$ is lower semicontinuous.
\item\label{Hypo-domain} There exists $\kappa > 0$ such that, for every $y \in Y$, there exists $x \in X$ with $\pi(x) = y$ and $L(x) \leq \kappa$.
\item\label{Hypo-compact} For every $\kappa > 0$, the set $\{x \in X \suchthat L(x) \leq \kappa\}$ is compact.
\item\label{Hypo-H-leq-L} There exists $C > 0$ such that $H(x, \widetilde x) \leq C (L(x) + L(\widetilde x) + 1)$ for every $(x, \widetilde x) \in X \times X$.
\end{hypothesisNAG}

Our main goals concerning $\NAGall$ are to show that \begin{enumerate*}[label={(\roman*)}] \item it is a potential game, i.e., its equilibria can be obtained as minimizers of a certain potential $\mathcal J$; and \item the potential $\mathcal J$ admits a minimizer, and hence $\NAGall$ admits an equilibrium.\end{enumerate*}

Let us introduce some notation to be used in the sequel. We define $J\colon X \times X \to \overline{\mathbb R}_+$ by
\begin{equation}
\label{eq:def-J}
J(x, \widetilde x) = L(x) + L(\widetilde x) + H(x, \widetilde x),
\end{equation}
and we define the functions $\mathcal L\colon \mathcal P(X) \to \overline{\mathbb R}_+$, $\mathcal H\colon \mathcal P(X) \times \mathcal P(X) \to \overline{\mathbb R}_+$, and $\mathcal J\colon \mathcal P(X) \to \overline{\mathbb R}_+$ by
\begin{subequations}
\label{eq:def-mathcal}
\begin{align}
\mathcal L(Q) & = \int_X L(x) \diff Q(x), \label{eq:def-mathcal-L} \\
\mathcal H(Q, \widetilde Q) & = \int_{X \times X} H(x, \widetilde x) \diff(Q \otimes \widetilde Q)(x, \widetilde x), \label{eq:def-mathcal-H} \\
\mathcal J(Q) & = \int_{X \times X} J(x, \widetilde x) \diff(Q \otimes Q)(x, \widetilde x). \label{eq:def-mathcal-J}
\end{align}
\end{subequations}
Given $m_0 \in \mathcal P(Y)$, we also set
\begin{equation}
\label{eq:def-Pm0X}
\mathcal P_{m_0}(X) = \{Q \in \mathcal P(X) \suchthat \pi_{\#} Q = m_0\}.
\end{equation}

\begin{remark}
\label{remk:Pm0X-closed}
Under \ref{Hypo-pi}, the set $\mathcal P_{m_0}(X)$ is a closed subset of $\mathcal P(X)$.
\end{remark}

\begin{remark}\label{Rem: link-cal L,H,J and F}
According to the definitions of $F$, $\mathcal{L}$, $\mathcal{H}$, and $\mathcal{J}$, from \eqref{eq:abstract-cost} and \eqref{eq:def-mathcal}, one can easily check that $\mathcal{J}(Q) = 2\mathcal{L}(Q) + \mathcal{H}(Q, Q)$ and $\int_{X} F(x, Q) \diff Q(x) = \mathcal{L}(Q) + \mathcal{H}(Q, Q)$, for every $Q \in \mathcal{P}(X)$. In particular, $\mathcal J(Q) < +\infty$ if and only if $\int_{X} F(x, Q) \diff Q(x) < +\infty$.
\end{remark}

\begin{remark}
\label{remk:not-convex}
The function \(\mathcal J\) is not necessarily convex and, in particular, it is not convex for the models motivating this work. Indeed, since \(Q \mapsto \mathcal L(Q)\) depends linearly on \(Q\), \(\mathcal J\) is convex if and only if \(Q \mapsto \mathcal H(Q, Q)\) is convex. A straightforward computation shows that, for every \(Q_1\) and \(Q_2\) in \(\mathcal P(X)\) and \(\zeta \in [0, 1]\), we have
\begin{multline*}
\mathcal H(\zeta Q_1 + (1 - \zeta) Q_2, \zeta Q_1 + (1 - \zeta) Q_2) - \zeta \mathcal H(Q_1, Q_1) - (1 - \zeta) \mathcal H(Q_2, Q_2) \\ = - \zeta (1 - \zeta) \int_{X \times X} H(x, \widetilde x) \diff((Q_1 - Q_2) \otimes (Q_1 - Q_2)) (x, \widetilde x),
\end{multline*}
so that \(\mathcal J\) is convex if and only if
\begin{equation}
\label{eq:cns-convex}
\int_{X \times X} H(x, \widetilde x) \diff((Q_1 - Q_2) \otimes (Q_1 - Q_2)) (x, \widetilde x) \geq 0
\end{equation}
for every \(Q_1\) and \(Q_2\) in \(\mathcal P(X)\), and \(\mathcal J\) is strictly convex if and only if
\begin{equation}
\label{eq:cns-strict-convex}
\int_{X \times X} H(x, \widetilde x) \diff((Q_1 - Q_2) \otimes (Q_1 - Q_2)) (x, \widetilde x) > 0
\end{equation}
for every \(Q_1\) and \(Q_2\) in \(\mathcal P(X)\) with \(Q_1 \neq Q_2\).

Assuming \(H\) to be defined as in Remark~\ref{remk:concrete} and the function \(h\) to be that of \eqref{eq:intro-cost-CS}, we have in particular that \(H(x, x) = 0\) for every \(x \in X\) with \(L(x) < +\infty\). Assuming that \(H\) takes other values than \(0\) and \(+\infty\), there exist \(x_1\) and \(x_2\) in \(X\) with \(x_1 \neq x_2\) such that \(H(x_1, x_2) \in (0, +\infty)\), and in particular \(L(x_1)\) and \(L(x_2)\) are both finite. Then, using \ref{Hypo-H-symmetric}, we have, for \(Q_1 = \delta_{x_1}\) and \(Q_2 = \delta_{x_2}\),
\[
\int_{X \times X} H(x, \widetilde x) \diff((Q_1 - Q_2) \otimes (Q_1 - Q_2)) (x, \widetilde x) = - 2 H(x_1, x_2) < 0,
\]
so that \eqref{eq:cns-convex} is not satisfied, and hence \(\mathcal J\) is not convex.

Note that \eqref{eq:cns-convex} and \eqref{eq:cns-strict-convex} are positive (semi-)definite-type assumptions. Indeed, when \(X\) is the finite set \(\{1, \dotsc, N\}\), the function \(H\) can be identified with the symmetric matrix \((H(i, j))_{i, j \in \{1, \dotsc, N\}}\) (which we also denote by \(H\)), and thus \eqref{eq:cns-convex} (respectively, \eqref{eq:cns-strict-convex}) reduces to \(\scalprod{q, H q} \geq 0\) for every \(q \in \mathbb R^N\) (respectively, \(\scalprod{q, H q} > 0\) for every \(q \in \mathbb R^N \setminus \{0\}\)) with \(\sum_{k = 1}^N q_k = 0\).
\end{remark}

\subsection{First properties}
\label{sec:first-properties}

Let us now establish some elementary properties of the functions introduced in \eqref{eq:def-mathcal}. We start with the following consequence of \ref{Hypo-H-leq-L}, whose proof is immediate.

\begin{lemma}
\label{lemm:H-leq-L}
Assume that \ref{Hypo-XY-Polish}, \ref{Hypo-LH}, and \ref{Hypo-H-leq-L} are satisfied. Let $C > 0$ be the constant from \ref{Hypo-H-leq-L} and $\mathcal L$ and $\mathcal H$ be defined as in \eqref{eq:def-mathcal-L} and \eqref{eq:def-mathcal-H}. Then, for every $(Q, \widetilde Q) \in \mathcal P(X) \times \mathcal P(X)$, we have
\[\mathcal H(Q, \widetilde Q) \leq C(\mathcal L(Q) + \mathcal L(\widetilde Q) + 1).\]
\end{lemma}

As a consequence of Remark~\ref{Rem: link-cal L,H,J and F} and Lemma~\ref{lemm:H-leq-L}, we also immediately obtain the following result.

\begin{corollary}
\label{coro:dom-L-J}
Assume that \ref{Hypo-XY-Polish}, \ref{Hypo-LH}, and \ref{Hypo-H-leq-L} are satisfied. Let $J$, $\mathcal L$, and $\mathcal J$ be defined as in \eqref{eq:def-J}, \eqref{eq:def-mathcal-L} and \eqref{eq:def-mathcal-J}. Then $\dom \mathcal L = \dom \mathcal J$.
\end{corollary}

\begin{remark}\label{Rem: extend H on span}
From Lemma~\ref{lemm:H-leq-L}, we deduce that the function $\mathcal{H}\colon\dom\mathcal{L} \times \dom\mathcal{L} \to \mathbb R_+$ can be extended by bilinearity to a unique bilinear function, still denoted by $\mathcal{H}$, defined in $\Span(\dom\mathcal{L}) \allowbreak \times \Span(\dom\mathcal{L})$ and taking values in $\mathbb R$, where the linear span is taken in the space of signed measures on $X$.
\end{remark}

In order to provide additional properties of $\dom \mathcal J$, we first establish the following preliminary result, which states that, in \ref{Hypo-domain}, the element $x \in X$ can be selected as a Borel-measurable function of $y \in Y$.

\begin{lemma}\label{PhiBorel}
Assume that \ref{Hypo-XY-Polish}--\ref{Hypo-LH}, \ref{Hypo-domain}, and \ref{Hypo-compact} are satisfied and let $\kappa > 0$ be as in \ref{Hypo-domain}. Then there exists a Borel measurable function $\Phi\colon Y \to X$ such that, for every $y \in Y$, we have $\pi(\Phi(y)) = y$ and $L(\Phi(y)) \leq \kappa$.
\end{lemma}

\begin{proof}
Let us consider the set valued map $\mathcal G\colon Y \rightrightarrows X$ defined by
\[
\mathcal{G}(y) = \bigl\{x \in X \mathrel{\big\vert} \pi(x) = y \text{ and } L(x) \leq \kappa  \bigr\}.
\]
Note that, by \ref{Hypo-domain}, we have $\mathcal G(y) \neq \emptyset$ for every $y \in Y$. Since $\pi$ is continuous and $L$ is lower semicontinuous, one immediately deduces that the graph of $\mathcal G$ is closed and, in particular, $\mathcal G(y)$ is closed for every $y \in Y$. In addition, $\mathcal G$ takes values in the set $\{x \in X \suchthat L(x) \leq \kappa\}$, which is compact thanks to \ref{Hypo-compact}. Hence, by \cite[Proposition~1.4.8]{Aubin2009Set}, we get that $\mathcal{G}$ is upper semicontinuous and, by \cite[Proposition~1.4.4]{Aubin2009Set}, the set $\mathcal{G}^{-1}(A) = \{y \in Y \suchthat \mathcal{G}(y) \cap A \neq \emptyset\}$ is closed for every closed set $A \subset X$. Hence, by \cite[Proposition~III.11]{Castaing1977Convex}, $\mathcal G^{-1}(B)$ is open for every open set $B \subset X$, and thus $\mathcal{G}$ is measurable. Now, by applying \cite[Theorem~8.1.3]{Aubin2009Set}, we conclude that the set-valued map $\mathcal{G}$ admits a Borel-measurable selection, which is our desired function $\Phi$. 
\end{proof}

\begin{remark}
Notice that it is not possible to apply \cite[Theorem~8.1.4]{Aubin2009Set} in order to extract directly a Borel measurable selection from the fact that $\mathcal G$ has a closed graph since the measure $m_0$ is not complete in the measurable space $\overline \Omega$ endowed with its Borel $\sigma$-algebra.  
\end{remark}

As a consequence of Lemma~\ref{PhiBorel}, we deduce the following result on $\dom \mathcal J$.

\begin{corollary}
\label{coro:dom-J-convex}
Assume that \ref{Hypo-XY-Polish}--\ref{Hypo-LH} and \ref{Hypo-domain}--\ref{Hypo-H-leq-L} are satisfied, and let $J$ and $\mathcal J$ be given by \eqref{eq:def-J} and \eqref{eq:def-mathcal-J}. Then, for every $m_0 \in \mathcal P(Y)$, the set $\dom \mathcal J \cap \mathcal P_{m_0}(X)$ is nonempty and convex.
\end{corollary}

\begin{proof}
By Corollary~\ref{coro:dom-L-J}, it suffices to show that $\dom \mathcal L \cap \mathcal P_{m_0}(X)$ is nonempty and convex, where $\mathcal L$ is given by \eqref{eq:def-mathcal-L}. Let $\Phi\colon Y \to X$ be the map from Lemma~\ref{PhiBorel} and define $Q = \Phi_{\#} m_0$. Then $\pi_{\#} Q = m_0$ and an immediate computation shows that $\mathcal L(Q) \leq \kappa < +\infty$, where $\kappa$ is the constant from \ref{Hypo-domain}. Hence $Q \in \dom \mathcal L \cap \mathcal P_{m_0}(X)$.

Note that $\mathcal P_{m_0}(X)$ is clearly convex from \eqref{eq:def-Pm0X}, and $\mathcal L$ is linear, implying that $\mathcal{L}(t Q + (1-t)\widetilde Q) = t\mathcal{L}(Q) + (1-t) \mathcal{L}(\widetilde Q)$ for every $Q, \widetilde Q \in \mathcal P(X)$ and $t \in [0, 1]$. Hence $\dom \mathcal L$ is convex, yielding that $\dom \mathcal L \cap \mathcal P_{m_0}(X)$ is convex.
\end{proof}

We will also need in the sequel the following property of the function $F$ from \eqref{eq:abstract-cost}.

\begin{lemma}
\label{lemm:F-lsc}
Assume that \ref{Hypo-XY-Polish}, \ref{Hypo-LH}, and \ref{Hypo-J} are satisfied, and let $\mathcal L$ be the function defined in \eqref{eq:def-mathcal-L}. For every $Q \in \dom \mathcal L$, the function $x \mapsto F(x, Q)$ is lower semicontinuous.
\end{lemma}

\begin{proof}
Note that, defining $J$ as in \eqref{eq:def-J}, we have
\[
F(x, Q) + \mathcal L(Q) = \int_X J(x, \widetilde x) \diff Q(\widetilde x),
\]
and thus, since $\mathcal L(Q) < +\infty$, it suffices to prove the lower semicontinuity of the function $x \mapsto \int_X J(x, \widetilde x) \diff Q(\widetilde x)$.

Let $(x_n)_{n \in \mathbb N}$ be a sequence in $X$ converging to some $x \in X$. Since $J$ is lower semicontinuous thanks to \ref{Hypo-J}, we have, for every $\widetilde x \in X$,
\[
J(x, \widetilde x) \leq \liminf_{n \to +\infty} J(x_n, \widetilde x).
\]
Then, Fatou's lemma shows that
\[
\int_X J(x, \widetilde x) \diff Q(\widetilde x) \leq \int_X \liminf_{n \to +\infty} J(x_n, \widetilde x) \diff Q(\widetilde x) \leq \liminf_{n \to +\infty} \int_X J(x_n, \widetilde x) \diff Q(\widetilde x),
\]
yielding the conclusion.
\end{proof}

Note that Lemma~\ref{lemm:F-lsc} also yields the following result on the existence of optimizers for $F(\cdot, Q)$.

\begin{corollary}
\label{coro:exist-optimal-FxQ}
Assume that \ref{Hypo-XY-Polish}--\ref{Hypo-LH} and \ref{Hypo-J}--\ref{Hypo-H-leq-L} are satisfied and let $\mathcal L$ be the function defined in \eqref{eq:def-mathcal-L}. For every $Q \in \dom \mathcal L$ and $y \in Y$, there exists $x \in X$ with $\pi(x) = y$ and such that
\[
F(x, Q) = \inf_{\substack{z \in X \\ \pi(z) = y}} F(z, Q).
\]
\end{corollary}

\begin{proof}
Let $(x_n)_{n \in \mathbb N}$ be a minimizing sequence for $F(\cdot, Q)$, i.e., $(x_n)_{n \in \mathbb N}$ is a sequence in $X$ with $\pi(x_n) = y$ for every $n \in \mathbb N$ and
\begin{equation}
\label{eq:minimizing-sequence-for-F}
\lim_{n \to +\infty} F(x_n, Q) = \inf_{\substack{z \in X \\ \pi(z) = y}} F(z, Q).
\end{equation}
By \ref{Hypo-domain}, there exists $z \in X$ with $\pi(z) = y$ such that $L(z) \leq \kappa$, where $\kappa$ is the constant from \ref{Hypo-domain}. Then, by \ref{Hypo-H-leq-L}, we have $F(z, Q) \leq (C + 1) L(z) + C \mathcal L(Q) + C \leq (C + 1) \kappa + C \mathcal L(Q) + C$, where $C$ is the constant from \ref{Hypo-H-leq-L}. Hence, the limit in the left-hand side in \eqref{eq:minimizing-sequence-for-F} is finite, and we assume, up to removing finitely many elements from the sequence, that $F(x_n, Q) < +\infty$ for every $n \in \mathbb N$, and thus the sequence $(F(x_n, Q))_{n \in \mathbb N}$ is bounded. Thus, the sequence $(L(x_n))_{n \in \mathbb N}$ is also bounded and, by \ref{Hypo-compact}, this implies that, up to extracting a subsequence, there exists $x \in X$ such that $x_n \to x$ as $n \to +\infty$. By \ref{Hypo-pi}, we obtain that $\pi(x) = y$ and, by \eqref{eq:minimizing-sequence-for-F} and Lemma~\ref{lemm:F-lsc}, we obtain that
\[
F(x, Q) \leq \lim_{n \to +\infty} F(x_n, Q) = \inf_{\substack{z \in X \\ \pi(z) = y}} F(z, Q) \leq F(x, Q),
\]
yielding the conclusion.
\end{proof}

\subsection{Potential game}
\label{sec:potential}

Our next goal is to prove that $\NAGall$ is a potential game with potential given by the function $\mathcal J$ from \eqref{eq:def-mathcal-J}. For games with finitely many players, a potential is a function such that, when a single player changes their strategy, the variation in the cost of that player is equal to (or proportional to) the variation of the potential. In particular, in potential games, equilibria can be found as minimizers of the potential. To adapt these ideas to our setting with a continuum of players, one needs to provide a suitable notion of ``variation'' of the potential $\mathcal J$ as a player changes their trajectory. It turns out that the good notion is that of differentiability of $\mathcal J$, in the sense of the following definition.

\begin{definition}
\label{def:deriv-J}
Assume that \ref{Hypo-LH} is satisfied and let $\mathcal J$ be defined by \eqref{eq:def-mathcal-J}. Given $Q_0 \in \dom\mathcal J$, we say that $\mathcal J$ is \emph{differentiable at $Q_0$} if there exists a Borel-measurable function $G\colon X \to \overline{\mathbb R}$ such that, for every $Q \in \dom \mathcal J$, $G$ is $(Q - Q_0)$-integrable and
\begin{equation}
\label{J is differentiable}
\lim_{t \to 0^+} \frac{\mathcal J(Q_0 + t(Q - Q_0)) - \mathcal J(Q_0)}{t} = \int_X G(x) \diff (Q - Q_0)(x).
\end{equation}
In this case, we denote $G$ by $\frac{\delta\mathcal J}{\delta Q}(Q_0)$ and we write the right-hand side of \eqref{J is differentiable} as $\scalprod{\frac{\delta\mathcal J}{\delta Q}(Q_0),\allowbreak Q - Q_0}$.
\end{definition}

\begin{remark}
\label{remk:deriv-measure}
There are several notions of derivative for functions defined in the space of probability measures, and the one presented in Definition~\ref{def:deriv-J} resembles the one from \cite[Definition~2.2.1]{Cardaliaguet2019Master} and \cite[Definition~5.43]{Carmona2018ProbabilisticI}, which is widely used in the analysis of mean field games. The main difference is that the notion we provide above requires less regularity of $(Q_0, x) \mapsto \frac{\delta\mathcal J}{\delta Q}(Q_0)(x)$, which is not necessarily continuous, and might even fail to be defined for some $Q_0 \in \dom \mathcal J$. This weaker notion suitably fits our purposes in the sequel. We refer the interested reader to \cite[Chapter~10]{Ambrosio2005Gradient}, \cite[Section~2.2.3 and Appendices~A.1 and A.2]{Cardaliaguet2019Master}, and \cite[Chapter~5]{Carmona2018ProbabilisticI} for other notions of derivation for functions defined in the space of probability measures and further discussion on the relations among those notions.
\end{remark}

Note that, if $G$ is a Borel-measurable function satisfying \eqref{J is differentiable}, then $G + \lambda$ also satisfies \eqref{J is differentiable} for any constant $\lambda \in \mathbb R$, since the integral of a constant with respect to the measure $Q - Q_0$ is zero. Hence, $\frac{\delta\mathcal J}{\delta Q}$ is not uniquely defined. We prove, however, in our next result, that, similarly to the analogous notion of derivative from \cite[Definition~2.2.1]{Cardaliaguet2019Master} and \cite[Definition~5.43]{Carmona2018ProbabilisticI}, $\frac{\delta\mathcal J}{\delta Q}$ is indeed unique up to an additive constant.

\begin{lemma}\label{lem: uniquness up to constant}
Assume that \ref{Hypo-XY-Polish}--\ref{Hypo-LH}, \ref{Hypo-domain}, and \ref{Hypo-compact} are satisfied and let $\mathcal J$ be defined by \eqref{eq:def-mathcal-J}. Let $Q_0 \in \dom \mathcal J$ and assume that $G_1$ and $G_2$ are two Borel measurable functions such that, for every $Q \in \dom \mathcal J$, $G_1$ and $G_2$ are $(Q - Q_0)$-integrable and
\[
\int_{X} G_1(x) \diff (Q - Q_0)(x) = \int_{X} G_2(x) \diff (Q - Q_0)(x).
\]
Then there exists $\lambda \in \mathbb R$ such that $G_1(x) = G_2(x) + \lambda$ for $Q_0$-almost every $x \in X$.
\end{lemma}

\begin{proof}
Let $\mathcal L$ be the function defined in \eqref{eq:def-mathcal-L}. Note that, for $Q_0$-almost every $x \in X$, we have $\delta_x \in \dom \mathcal J$. Indeed, since $Q_0 \in \dom \mathcal J = \dom \mathcal L$ by Corollary~\ref{coro:dom-L-J}, we deduce that $L(x) < +\infty$ for $Q_0$-almost every $x \in X$, and thus $\mathcal L(\delta_x) = L(x) < +\infty$, showing that $\delta_x \in \dom \mathcal L = \dom \mathcal J$.

We now set $\lambda = \int_{X} (G_1 - G_2)(x) \diff Q_0(x)$ and observe that, by assumption, for every $Q \in \dom\mathcal J$, we have
\[
\int_{X} G_1(x) \diff Q(x) = \int_{X} G_2(x) \diff Q(x) + \lambda,
\]
and hence, for $Q_0$-almost every $x \in X$, we can take $Q = \delta_x$ in the above identity to deduce that $G_1(x) = G_2(x) + \lambda$, as required.
\end{proof}

Our next result is an important step in formulating $\NAGall$ as a potential non-atomic game with potential $\mathcal J$, since, under the additional assumption \ref{Hypo-H-symmetric}, it relates the variations of $\mathcal J$, in the sense of Definition~\ref{def:deriv-J}, with the individual cost $F(x, Q)$ of an agent from \eqref{eq:abstract-cost}.

\begin{theorem}\label{thm:nabla J is F}
Assume that \ref{Hypo-XY-Polish}--\ref{Hypo-H-symmetric}, \ref{Hypo-domain}, and \ref{Hypo-compact} are satisfied and let $\mathcal J$ be defined by \eqref{eq:def-mathcal-J}. Then, for every $Q_0 \in \dom\mathcal{J}$, the function $\mathcal{J}$ is differentiable at $Q_0$, with $\frac{\delta \mathcal{J}}{\delta Q}(Q_0)(x) = 2 F(x, Q_0)$ for $Q_0$-almost every $x \in X$.
\end{theorem}

\begin{proof}
Fix $Q_0 \in \dom\mathcal J$ and let $\mathcal L$ and $\mathcal H$ be defined as in \eqref{eq:def-mathcal-L} and \eqref{eq:def-mathcal-H}. For every $Q \in \dom\mathcal J$ and $t \in (0, 1]$, using Remark~\ref{Rem: link-cal L,H,J and F}, we have that
\begin{align*}
\frac{\mathcal{J}(Q_0 + t(Q - Q_0)) - \mathcal{J}(Q_0)}{t} = {} & 2 (\mathcal{L}(Q) + \mathcal{H}(Q_0, Q)) - 2 (\mathcal{L}(Q_0) + \mathcal{H}(Q_0, Q_0)) \\
& + t \mathcal{H}(Q - Q_0, Q - Q_0),
\end{align*}
where $\mathcal{H}(Q - Q_0, Q - Q_0) \in \mathbb R$ is well-defined by Remark~\ref{Rem: extend H on span}. Hence, letting $t \to 0^+$, we obtain from Remark~\ref{Rem: link-cal L,H,J and F} that
\begin{equation*}
\begin{aligned}
\lim_{t\to 0} \frac{\mathcal{J}(Q_0+ t(Q-Q_0))-\mathcal{J}(Q_0)}{t} &= 2 (\mathcal{L}(Q) + \mathcal{H}(Q_0, Q)) - 2 (\mathcal{L}(Q_0) + \mathcal{H}(Q_0, Q_0))\\
&= \int_{X} 2F(x, Q_0) \diff Q(x) - \int_{X} 2F(x, Q_0) \diff Q_0(x)\\
&=\int_{X} 2F(x, Q_0) \diff(Q - Q_0)(x),
\end{aligned}
\end{equation*}
yielding the conclusion.
\end{proof}

\begin{remark}
The equality \(\frac{\delta \mathcal{J}}{\delta Q}(Q_0)(x) = 2 F(x, Q_0)\) in the statement of Theorem~\ref{thm:nabla J is F} is an abuse of notation, since, as discussed above, \(\frac{\delta \mathcal{J}}{\delta Q}(Q_0)\) is well-defined only up to an additive constant. Such an equality should thus be understood as stating that the set of functions \(G\) satisfying \eqref{J is differentiable} is \(\{2 F(\cdot, Q_0) + \lambda \suchthat \lambda \in \mathbb R\}\). As it will be clear in the sequel, an additive constant plays no role in the characterization of equilibria of \(\NAGall\), and we will thus keep using such an abuse of notation for sake of simplicity.
\end{remark}

Now that we have established a link between variations of $\mathcal J$ and the individual cost $F(x, Q)$ in Theorem~\ref{thm:nabla J is F}, our next step to study the potential structure of $\NAGall$ is to relate equilibria of this game with minimizers of $\mathcal J$. For that purpose, we start with the following definition.

\begin{definition}
Assume that \ref{Hypo-LH} is satisfied and let $\mathcal J$ be defined by \eqref{eq:def-mathcal-J}. We say that $Q_0 \in \dom\mathcal{J}\cap \mathcal{P}_{m_0}(X)$ is a \emph{critical point of $\mathcal{J}$ in $\mathcal{P}_{m_0}(X)$} if $\mathcal J$ is differentiable at $Q_0$ and
\begin{equation}
\label{eq:Euler}
\scalprod*{\frac{\delta \mathcal{J}}{\delta Q}(Q_0), Q - Q_0} \geq 0 \qquad \text{for every }Q \in \dom\mathcal J \cap \mathcal P_{m_0}(X).
\end{equation}
\end{definition}

Inequalities of the form \eqref{eq:Euler} are sometimes known as \emph{Euler inequalities}, and they are usually necessary conditions for minimization of differentiable functions (see, e.g., \cite[Theorem~10.2.1 and Remark~10.2.2]{Allaire2007Numerical} for the case of functions defined on Hilbert spaces). Notice that, by Lemma~\ref{lem: uniquness up to constant}, the notion of critical point introduced in above is independent of the choice of representative of $\frac{\delta \mathcal J}{\delta Q}(Q_0)$, since adding a constant to $\frac{\delta \mathcal J}{\delta Q}(Q_0)$ does not affect the Euler inequality \eqref{eq:Euler}. Our next result states that \eqref{eq:Euler} is indeed a necessary condition for the minimization of $\mathcal J$ on $\mathcal P_{m_0}(X)$.

\begin{proposition}\label{prop:J InEqua opti}
Assume that \ref{Hypo-LH} satisfied and let $Q_0 \in \dom\mathcal J$ be a local minimizer of $\mathcal J$ in $\mathcal P_{m_0}(X)$, i.e., there exists a neighborhood $V$ of $Q_0$ in $\mathcal P_{m_0}(X)$, with the topology of weak convergence of probability measures, such that $\mathcal J(Q) \geq \mathcal J(Q_0)$ for every $Q \in V$. If $\mathcal J$ is differentiable at $Q_0$, then $Q_0$ is a critical point of $\mathcal{J}$ in $\mathcal{P}_{m_0}(X)$.
\end{proposition}

\begin{proof}
Let $Q_0 \in \dom\mathcal J$ be a local minimizer of $\mathcal J$ in $\mathcal P_{m_0}(X)$ and assume that $\mathcal J$ is differentiable at $Q_0$. Then, for every $Q \in \dom \mathcal J \cap \mathcal P_{m_0}(X)$, we deduce from the facts that $Q_0$ is a local minimizer of $\mathcal J$ in $\mathcal P_{m_0}(X)$ and $Q_0 + t(Q - Q_0) \to Q_0$ as $t \to 0^+$ in the topology of weak convergence in $\mathcal P_{m_0}(X)$ that
\[
\frac{\mathcal{J}(Q_0+ t(Q-Q_0))-\mathcal{J}(Q_0)}{t} \ge 0
\]
for small enough $t > 0$. Taking the limit as $t \to 0^+$ immediately yields \eqref{eq:Euler}.
\end{proof}

We next establish one of our main results on the potential structure of $\NAGall$, stating that equilibria of this game coincide with critical points of $\mathcal J$ in $\mathcal P_{m_0}(X)$. The proof we present here is closely based on that of \cite[Lemma~3.3]{Santambrogio2021Cucker}.

\begin{theorem}
\label{thm:potential}
Assume that \ref{Hypo-XY-Polish}--\ref{Hypo-H-leq-L} are satisfied and let $\mathcal J$ be given by \eqref{eq:def-mathcal-J}. A measure $Q_0 \in \mathcal P_{m_0}(X)$ is an equilibrium for $\NAGall$ if and only if it is a critical point of $\mathcal J$ in $\mathcal P_{m_0}(X)$.

As a consequence, any local minimizer of $\mathcal J$ in $\mathcal P_{m_0}(X)$ is an equilibrium for $\NAGall$.
\end{theorem}

\begin{proof}
Recall that, according to Remark~\ref{Rem: link-cal L,H,J and F}, $Q_0\in \dom \mathcal J$ is equivalent to having $\int_{X} F(x, \allowbreak Q_0)\allowbreak \diff Q_0(x) < +\infty$.

Assume first that $Q_0\in \mathcal P_{m_0}(X)$ is a critical point of $\mathcal{J}$ in $\mathcal P_{m_0}(X)$ and note that, by definition, we have $Q_0 \in \dom\mathcal J$, and hence $\int_{X} F(x, \allowbreak Q_0)\allowbreak \diff Q_0(x) < +\infty$. To obtain a contradiction, suppose that the set 
\[
\left\{x\in X \suchthat \exists z \in X \text{ such that } \pi(z) = \pi(x) \text{ and } F(z, Q_0) < F(x, Q_0)\right\} 
\]
is not $Q_0$-negligible. Note that the above set is equal to
\[
\bigcup_{\substack{(q, r) \in \mathbb Q_+^2 \\ 0 < q < r}} \left\{x\in X \suchthat F(x,Q_0) > r \text{ and } \{z \in X \suchthat \pi(z) = \pi(x) \text{ and } F(z, Q_0) \leq q\}\neq \emptyset\right\},
\]
and thus there exists a pair of positive rational numbers $(q,r)$ with $q<r$ such that the set $A$ defined by
\begin{equation}
\label{eq:defi-A}
A = \left\{x\in X \suchthat F(x,Q_0) > r \text{ and } \{z \in X \suchthat \pi(z) = \pi(x) \text{ and } F(z, Q_0) \leq q\}\neq \emptyset\right\}
\end{equation}
is not $Q_0$-negligible.

For any $p \in \mathbb R$, define $A^p$ as
\[
A^p=\left\{x \in X \suchthat F(x,Q_0) \leq p\right\}.
\]
Note that $A^p$ is closed since, by Lemma~\ref{lemm:F-lsc}, $x \mapsto F(x, Q_0)$ is lower semicontinuous. In addition, $A^p \subset \{x \in X \suchthat L(x) \leq p\}$ and the latter set is compact thanks to \ref{Hypo-compact}. Hence, $A^p$ is compact.

Let us now define the multifunction $\mathcal{S}$ as 
\begin{equation*}
\mathcal{S}\colon
\left\{
\begin{aligned}
 X & \rightrightarrows A^q, \\
 x & \mapsto \mathcal{S}(x) = \left\{z \in A^q \suchthat \pi(z)=\pi(x)\right\}.
\end{aligned}
\right.
\end{equation*}
Using \ref{Hypo-pi}, it is immediate to verify that $\mathcal S$ has closed graph and, using the fact that $A^q$ is compact, one can also easily verify that the domain $\dom \mathcal S$, defined as $\dom \mathcal{S} = \{x\in X\suchthat \mathcal{S}(x) \neq \emptyset\}$, is closed. It is also clearly nonempty, since $A^q \subset \dom \mathcal S$. By \cite[Proposition~1.4.8]{Aubin2009Set}, we deduce that $\mathcal{S}$ is upper semicontinuous. Hence, by \cite[Proposition~1.4.4]{Aubin2009Set}, $\mathcal{S}^{-1}(A) =\{x \in X \suchthat \mathcal{S}(x) \cap A \neq \emptyset\}$ is closed for every closed set $A$. Now by definition of measurability and \cite[Proposition~III.11]{Castaing1977Convex}, we obtain that $\mathcal{S}$ is Borel-measurable, and, by \cite[Theorem~8.1.3]{Aubin2009Set}, there exists a Borel-measurable function $s \colon \dom \mathcal{S} \to A^q$ satisfying
\[
s(x)\in \mathcal{S}(x), \qquad \text{ for all } x \in \dom \mathcal{S}. 
\]

Note that the set $A$ defined in \eqref{eq:defi-A} is Borel measurable, since it can be written as
\[
A = \left\{x\in X \suchthat F(x,Q_0)>r\right\} \cap \dom \mathcal S,
\]
$\dom \mathcal S$ is closed, and $\left\{x\in X \suchthat F(x,Q_0)>r\right\}$ is the complement of $A^r$, and hence it is open. We can thus define a measure $\widetilde{Q} \in \mathcal P(X)$ by $\widetilde{Q}=s_{\#}(Q_0|_A)+Q_0|_{(X \setminus A)}$, which is well-defined since $\widetilde Q(X) = Q_0(A \cap s^{-1}(X)) + Q_0(X \setminus A) = Q_0(A) + Q_0(X \setminus A) = 1$, and in addition $\widetilde Q \in \mathcal P_{m_0}(X)$, since
\[
\begin{aligned}
\pi_{\#} \widetilde Q &= \pi_{\#} \left(s_{\#}(Q_0|_A)+Q_0|_{(X \setminus A)}\right)
= \pi_{\#} \left(s_{\#}(Q_0|_A)\right) + \pi_{\#} (Q_0|_{(X \setminus A)})\\
&= (\pi \circ s)_{\#} \left(Q_0|_A\right) + \pi_{\#}(Q_0|_{(X \setminus A)}) 
= \pi_{\#}(Q_0|_{A}) + \pi_{\#}(Q_0|_{(X \setminus A)}) \\
&= \pi_{\#} (Q_0|_{A} + Q_0|_{(X \setminus A)}) = \pi_{\#} Q_0 = m_0,
\end{aligned}
\]
where we use that $\pi \circ s (x) = \pi(x)$ for every $x \in \dom\mathcal S$. Using that $Q_0(A) > 0$, we also compute
\[
\begin{aligned}
\int_{X} F(x,Q_0) \diff \widetilde Q(x) 
&= \int_{X \setminus A} F(x, Q_0)\diff Q_0(x) + \int_{A} F(s(x), Q_0) \diff Q_0(x)\\
&\leq \int_{X \setminus A} F(x, Q_0)\diff Q_0(x) + q Q_0(A)\\
&< \int_{X \setminus A} F(x, Q_0)\diff Q_0(x) + r Q_0(A)\\
&< \int_{X \setminus A} F(x, Q_0)\diff Q_0(x) + \int_{A} F(x, Q_0) \diff Q_0(x)\\
&= \int_{X} F(x,Q_0) \diff  Q_0(x),
\end{aligned}
\]
which, using Theorem~\ref{thm:nabla J is F}, contradicts the fact that $Q_0$ is a critical point of $\mathcal J$ in $\mathcal P_{m_0}(X)$.

For the converse implication, assume that $Q_0\in \mathcal P_{m_0}(X)$ is an equilibrium of $\NAGall$. In particular, $\int_X F(x, Q_0) \diff Q_0(x) < +\infty$, and thus $Q_0 \in \dom \mathcal J$. Define $\nu\colon Y \to \overline{\mathbb R}_+$ by
\begin{equation}
\label{eq:defi-nu}
\nu(y) = \inf_{\substack{z\in X\\ \pi(z) = y}} F(z, Q_0)
\end{equation}
and note that $\nu$ is Borel-measurable. Indeed, for every $\rho \geq 0$, we have
\begin{align*}
\{y \in Y \suchthat \nu(y) \leq \rho\} & = \bigcap_{n \in \mathbb N^\ast} \left\{y \in Y \suchthat \exists z \in X \text{ such that } \pi(z) = y \text{ and } F(z, Q_0) \leq \rho + \frac{1}{n}\right\} \\
& = \bigcap_{n \in \mathbb N^\ast} \pi\left(\left\{z \in X \suchthat F(z, Q_0) \leq \rho + \frac{1}{n}\right\}\right).
\end{align*}
For every $n \in \mathbb N^\ast$, by Lemma~\ref{lemm:F-lsc} and \ref{Hypo-compact}, we have that $\left\{z \in X \suchthat F(z, Q_0) \leq \rho + \frac{1}{n}\right\}$ is a closed subset of the compact set $\left\{z \in X \suchthat L(z) \leq \rho + \frac{1}{n}\right\}$, and hence it is compact. Thus, by \ref{Hypo-pi}, the set $\pi\left(\left\{z \in X \suchthat F(z, Q_0) \leq \rho + \frac{1}{n}\right\}\right)$ is compact, and thus it is closed, allowing one to conclude that $\{y \in Y \suchthat \nu(y) \leq \rho\}$ is a Borel set. This yields the Borel-measurability of $\nu$, as required.

By definition of equilibrium, we have $F(x, Q_0) = \nu(\pi(x))$ for $Q_0$-almost every $x \in X$. Hence
\[
\int_{X} F(x,Q_0) \diff Q_0(x) = \int_{X} \nu(\pi(x)) \diff Q_0(x).
\]
On the other hand, if one takes $Q \in \dom \mathcal{J} \cap \mathcal P_{m_0}(X)$, then
\[
\begin{aligned}
\int_{X} F(z,Q_0) \diff Q(z) & \geq \int_{X} \nu(\pi(z)) \diff Q(z) = \int_{Y} \nu(y) \diff (\pi_{\#}Q)(y) = \int_Y \nu(y) \diff m_0(y) \\
& = \int_Y \nu(y) \diff (\pi_{\#} Q_0)(y) = \int_{X} \nu(\pi(x)) \diff Q_0(x) = \int_{X} F(x,Q_0) \diff Q_0(x),
\end{aligned}
\]
and thus, using Theorem~\ref{thm:nabla J is F}, we deduce that $Q_0$ is a critical point of $\mathcal J$ in $\mathcal P_{m_0}(X)$.

Finally, the last assertion of the statement follows as an immediate consequence of Proposition~\ref{prop:J InEqua opti}.
\end{proof}

\subsection{Existence of equilibria}
\label{sec:existence-of-equilibria}

We now turn to the question of existence of an equilibrium for $\NAGall$ which, thanks to the potential structure of this game highlighted in Section~\ref{sec:potential}, can be addressed by studying the existence of minimizers for the function $\mathcal J$ from \eqref{eq:def-mathcal-J}. We start with the following preliminary result on $\mathcal J$.

\begin{lemma}
\label{lemm:mathcal-J-lsc}
Assume that \ref{Hypo-XY-Polish} and \ref{Hypo-J} are satisfied and let $\mathcal J$ be the function defined in \eqref{eq:def-mathcal-J}. Then $\mathcal J$ is lower semicontinuous on $\mathcal P(X)$.
\end{lemma}

\begin{proof}
Let $(Q_n)_{n \in \mathbb N}$ be a sequence in $\mathcal P(X)$ converging to some $Q \in \mathcal P(X)$.
Since $X$ is a Polish space, it follows from \cite[Lemma~A.1]{Santambrogio2021Cucker} that the sequence of the product measures $(Q_{n}\otimes Q_{n})_{n\in \mathbb N}$ converges to $Q \otimes Q$. For $C > 0$, let $J_C \colon X \times X \to \mathbb R_+$ be defined by $J_C(x, \widetilde x) =\min\{J(x, \widetilde x), C\}$ and notice that, by \ref{Hypo-J}, the function $J_C$ is lower semicontinuous and boun\-ded. Then, by \cite[Corollary~8.2.5]{Bogachev2007Measure}, we have
\[
\begin{aligned}
\int_{X \times X} J_C(x, \widetilde x) \diff(Q\otimes Q)(x, \widetilde x) & \leq \liminf_{n\to +\infty} \int_{X \times X} J_C(x, \widetilde x)\diff(Q_{n} \otimes Q_{n})(x, \widetilde x) \\
&\leq \liminf_{n\to +\infty} \int_{X \times X} J(x, \widetilde x)\diff(Q_{n} \otimes Q_{n})(x, \widetilde x) = \liminf_{n \to +\infty} \mathcal J(Q_n).
\end{aligned}
\]
Since the above holds for every $C>0$, by Lebesgue's monotone convergence theorem, we deduce that
\[
\mathcal J(Q) = \int_{X \times X} J(x, \widetilde x) \diff(Q\otimes Q)(x, \widetilde x) = \lim_{C \to +\infty} \int_{X \times X} J_C(x, \widetilde x) \diff (Q\otimes Q)(x, \widetilde x) \leq \liminf_{n \to +\infty} \mathcal J(Q_n),
\]
yielding the conclusion.
\end{proof}

Our main result on the existence of minimizers for $\mathcal J$ and of equilibria for $\NAGall$ is the following

\begin{theorem}
\label{thm:exist-abstract}
Assume that \ref{Hypo-XY-Polish}--\ref{Hypo-LH} and \ref{Hypo-J}--\ref{Hypo-H-leq-L} are satisfied and let $\mathcal J$ be the function from \eqref{eq:def-mathcal-J}. Then, for every $m_0 \in \mathcal P(Y)$, $\mathcal J$ admits a minimizer in $\mathcal P_{m_0}(X)$.

As a consequence, if in addition \ref{Hypo-H-symmetric} is also satisfied, then there exists an equilibrium $Q$ for $\NAGall$.
\end{theorem}

\begin{proof}
Let $m_0 \in \mathcal P(Y)$ and note that, by Corollary~\ref{coro:dom-J-convex}, we have $\dom \mathcal J \cap \mathcal P_{m_0}(X) \neq \emptyset$. Let $(Q_n)_{n \in \mathcal N}$ be a minimizing sequence for $\mathcal J$ in $\mathcal P_{m_0}(X)$, i.e., $(Q_n)_{n \in \mathbb N}$ is a sequence in $\dom\mathcal J \cap \mathcal P_{m_0}(X)$ and
\[\lim_{n \to +\infty} \mathcal J(Q_n) = \inf_{Q \in \mathcal P_{m_0}(X)} \mathcal J(Q).\]

We claim that the sequence $(Q_n)_{n \in \mathbb N}$ is tight. Indeed, by Markov's inequality, for every $M > 0$, we have
\[
Q_n(\{x \in X \suchthat L(x) > M\}) \leq \frac{1}{M} \mathcal{L}(Q_n) \leq \frac{1}{M} \mathcal{J}(Q_n),
\]
and the conclusion follows since $(\mathcal{J}(Q_n))_{n \in \mathbb N}$ is a bounded sequence and, by \ref{Hypo-compact}, $\{x \in X \suchthat L(x) \leq M\}$ is compact for every $M > 0$. Therefore, by Prokhorov's theorem (see, e.g., \cite[Theorem~5.1.3]{Ambrosio2005Gradient}), up to extracting a subsequence (which we still denote by $(Q_n)_{n \in \mathbb N}$ for simplicity), there exists $Q \in \mathcal P(X)$ such that $Q_n \to Q$ as $n \to +\infty$. In addition, by Remark~\ref{remk:Pm0X-closed}, we have $Q \in \mathcal P_{m_0}(X)$. By Lemma~\ref{lemm:mathcal-J-lsc}, we deduce that $\mathcal J(Q) \leq \lim_{n \to +\infty} \mathcal J(Q_n) = \inf_{\widetilde Q \in \mathcal P_{m_0}(X)} \mathcal J(\widetilde Q)$, and thus $Q$ is a minimizer of $\mathcal J$ in $\mathcal P_{m_0}(X)$.

The last part of the statement is an immediate consequence of Theorem~\ref{thm:potential}.
\end{proof}

\subsection{Strong equilibria}
\label{sec:strong-equilibria}

Definitions~\ref{def:equilibrium-3} and \ref{def:equilibrium-potential} of equilibria of the mean field game from Section~\ref{Cucker-Smale:The Model} and of $\NAGall$ require an optimization problem to be solved by \emph{$Q$-almost every} element of the space, similarly to other Lagrangian approaches in the literature, such as those from \cite{Mazanti2019Minimal, Sadeghi2022Multi, Sadeghi2022Nonsmooth}. An alternative definition consists in requiring the optimization problem to be solved \emph{for every} element in the support of $Q$. This stronger notion of equilibrium was used, for instance, in \cite{Santambrogio2021Cucker, Cannarsa2018Existence, Cannarsa2019C11, Cannarsa2021Mean, Dweik2020Sharp}, and, in many situations, as shown in \cite[Remark~4.6]{Dweik2020Sharp} and \cite[Proposition~3.7]{Sadeghi2022Nonsmooth}, both notions of equilibria coincide. We aim at proving that this is also the case in our setting, at the cost of an additional assumption on $\NAGall$. To do that, let us start by defining the notion of strong equilibrium.

\begin{definition}
\label{def:strong-equilibrium-potential}
Let $X$ and $Y$ be Polish spaces, $\pi\colon X \to Y$, $L\colon X \to \overline{\mathbb R}_+$, and $H\colon X \times X \to \overline{\mathbb R}_+$ be Borel-measurable functions, and $m_0 \in \mathcal P(Y)$. We say that $Q \in \mathcal P(X)$ is a \emph{strong equilibrium} of $\NAGall$ if $\pi_{\#} Q = m_0$, $\int_X F(x, Q) \diff Q(x) < +\infty$, and every $x \in \supp(Q)$ solves $\Min(\pi(x), F)$ with $F$ given by \eqref{eq:abstract-cost}.
\end{definition}

To prove that the notions of equilibrium and strong equilibrium are equivalent, we will make use of an additional technical assumption, which we now state.

\begin{hypothesisNAG}[resume]
\item\label{Hypo-technical} For every $Q \in \dom \mathcal L$, the set $\OOpt(Q)$ defined by
\begin{equation}
\label{eq:def-OOpt}
\OOpt(Q) = \left\{x \in X \suchthat F(x, Q) = \inf_{\substack{\tilde x \in X \\ \pi(\tilde x) = \pi(x)}} F(\tilde x, Q)\right\}
\end{equation}
is closed.
\end{hypothesisNAG}

As a consequence, we obtain the following result on the equivalence between equilibrium and strong equilibrium for $\NAGall$.

\begin{theorem}
\label{thm:strong-iff-equilibrium}
Assume that \ref{Hypo-XY-Polish}--\ref{Hypo-LH}, \ref{Hypo-J}, and \ref{Hypo-technical} are satisfied. Then $Q \in \mathcal P_{m_0}(X)$ is an equilibrium of the non-atomic game $\NAGall$ if and only if it is a strong equilibrium of this game.
\end{theorem}

\begin{proof}
It is immediate from Definitions~\ref{def:equilibrium-potential} and \ref{def:strong-equilibrium-potential} that every strong equilibrium is also an equilibrium of $\NAGall$. To prove the converse statement, let $Q \in \mathcal P_{m_0}(X)$ be an equilibrium of $\NAGall$ and note that, by Remark~\ref{Rem: link-cal L,H,J and F}, we have $Q \in \dom\mathcal J$, implying that $Q \in \dom\mathcal L$. Note also that, since $Q$ is an equilibrium of $\NAGall$, we have $Q(\OOpt(Q)) = 1$, where $\OOpt(Q)$ is the set defined in \eqref{eq:def-OOpt}.

Let $x \in \supp(Q)$. Then, by definition of support and using the fact that $Q(\OOpt(Q)) = 1$, there exists a sequence $(x_n)_{n \in \mathbb N}$ in $\OOpt(Q)$ such that $x_n \to x$ as $n \to +\infty$. Since $\OOpt(Q)$ is closed by \ref{Hypo-technical}, we deduce that $x \in \OOpt(Q)$. Thus $\supp(Q) \subset \OOpt(Q)$, concluding the proof.
\end{proof}

\begin{remark}
\label{remk:new-proof}
If one also assumes \ref{Hypo-technical} in Theorem~\ref{thm:potential}, then one can provide a simpler argument for the first implication in its proof, i.e., the fact that critical points of $\mathcal J$ in $\mathcal P_{m_0}(X)$ are equilibria of $\NAG(X, Y, \pi, L, H, m_0)$.

Indeed, notice first that, under the assumptions of Theorem~\ref{thm:potential} and \ref{Hypo-technical}, for every $Q_0 \in \dom \mathcal L$, we have
\begin{equation}
\label{eq:OOpt-compact}
\OOpt(Q_0) \subset \{x \in X \suchthat L(x) \leq (C+1)\kappa + C \mathcal L(Q_0) + C\},
\end{equation}
where $\kappa$ and $C$ are the constants from \ref{Hypo-domain} and \ref{Hypo-H-leq-L} and $\OOpt(Q_0)$ is the set defined in \eqref{eq:def-OOpt}. This is the case since, given $x \in \OOpt(Q_0)$, we have $L(x) \leq F(x, Q_0) \leq F(z, Q_0)$ for every $z \in X$ with $\pi(z) = \pi(x)$. By \ref{Hypo-H-leq-L}, we have $F(z, Q_0) \leq (C+1) L(z) + C \mathcal L(Q_0) + C$ for every $z \in X$, and, combining with \ref{Hypo-domain}, we deduce that there exists $z \in X$ such that $\pi(z) = \pi(x)$ and $F(z, Q_0) \leq (C+1) \kappa + C \mathcal L(Q_0) + C$, yielding the inclusion \eqref{eq:OOpt-compact}. Now, by \ref{Hypo-compact} and \ref{Hypo-technical}, we deduce from \eqref{eq:OOpt-compact} that $\OOpt(Q_0)$ is compact.

Assume that $Q_0 \in \mathcal P_{m_0}(X)$ is a critical point of $\mathcal J$ in $\mathcal P_{m_0}(X)$. Let $\Opt\colon Y \rightrightarrows X$ be the set-valued map defined for $y \in Y$ by
\begin{equation*}
\Opt(y) = \left\{x \in X \suchthat \pi(x) = y \text{ and } F(x, Q_0) = \inf_{\substack{ z \in X\\ \pi(z)=y}} F(z, Q_0)\right\}.
\end{equation*}
Clearly, $\Opt(y) \subset \OOpt(Q_0)$ for every $y \in Y$, and it follows from Corollary~\ref{coro:exist-optimal-FxQ} that $\Opt(y) \neq \emptyset$ for every $y \in Y$. Using the closedness of $\OOpt(Q_0)$ and the continuity of $\pi$, it is easy to check that the graph of $\Opt$ is closed. Since $\OOpt(Q_0)$ is compact, \cite[Proposition~1.4.8]{Aubin2009Set} shows that $\OOpt$ is upper semicontinuous and, combining \cite[Proposition~1.4.4]{Aubin2009Set}, \cite[Proposition~III.11]{Castaing1977Convex}, and \cite[Theorem~8.1.3]{Aubin2009Set}, we conclude that the set valued map $\Opt$ admits a Borel-measurable selection, which we denote by $\phi\colon Y \to X$. 

Define $\widehat Q = \phi_{\#} m_0$ and notice that $\pi_{\#} \widehat Q = m_0$ and
\[
\int_{Y} \nu(y) \diff m_0(y) = \int_{X} F(x, Q_0) \diff \widehat Q(x),
\]
where $\nu$ is the function defined in \eqref{eq:defi-nu}. Since $Q_0$ is a critical point of $\mathcal J$ in $\mathcal P_{m_0}(X)$, we obtain, using Theorem~\ref{thm:nabla J is F}, that
\begin{equation} 
\label{eq:inequality-F-leq-nu-pi}
\int_{X} F(x, Q_0) \diff Q_0(x) \leq \int_{X} F(x, Q_0) \diff \widehat Q(x) = \int_{Y} \nu(y) \diff m_0(y) = \int_{X} \nu(\pi(x)) \diff Q_0(x).
\end{equation}
Let $\Theta\colon X \to \mathbb R$ be the function defined by $\Theta(x) = F(x, Q_0) - \nu(\pi(x))$. By definition, $\Theta(x) \geq 0$ for any $x \in X$. It follows from \eqref{eq:inequality-F-leq-nu-pi} that
\begin{equation*}
\int_{X} \Theta(x) \diff Q_0(x) \leq 0,
\end{equation*}
therefore, $\Theta(x) = 0$ for $Q_0$-a.e.\ $x$, which implies that $Q_0$ is an equilibrium of $\NAGall$.
\end{remark}

\section{Application to the mean field game model with pairwise interactions}
\label{sec:application}

We now come back to the mean field game model of Section~\ref{Cucker-Smale:The Model}, and our aim is to apply the results from Section~\ref{sec:abstract} for the non-atomic game $\NAGall$ to the former MFG. To do so, we let $X$, $Y$, $\pi$, $L$, and $H$ be defined from the MFG from Section~\ref{Cucker-Smale:The Model} as in Remark~\ref{remk:concrete} and we aim at proving that \ref{Hypo-XY-Polish}--\ref{Hypo-H-leq-L} are satisfied if we assume \ref{Hypo3-Omega}--\ref{individualCost}, and that the notions of equilibria from Definitions~\ref{def:equilibrium-3} and \ref{def:equilibrium-potential} coincide. Clearly, \ref{Hypo-XY-Polish} and \ref{Hypo-pi} are satisfied as soon as one assumes \ref{Hypo3-Omega} and \ref{Hypo3-Gamma}, \ref{Hypo-H-symmetric} follows from \ref{h convex and symmetric}\ref{item:H-symmetric}, and \ref{Hypo-domain} follows from \ref{individualCost}. We are thus left to show \ref{Hypo-LH}, \ref{Hypo-J}, \ref{Hypo-compact}, and \ref{Hypo-H-leq-L}. We start by the following result.

\begin{lemma}
\label{lemm:tau-psi-lsc}
Assume that \ref{Hypo3-Omega}, \ref{Hypo3-Gamma}, and \ref{Psi1} are satisfied and let $\tau$ be the function defined in \eqref{eq:defi3-tau}. Then the functions $\tau$ and $\Psi \circ \tau$ are lower semicontinuous.
\end{lemma}

\begin{proof}
Let $(\gamma_n)_{n\in \mathbb N}$ be a sequence in $\Gamma$ converging to some $\gamma \in \Gamma$ uniformly on compact time intervals. We want to prove that
\[
\liminf_{n\to \infty} \tau(\gamma_n) \ge \tau(\gamma).
\]
Let $\tau_\ast$ denote the left-hand side of the above inequality. If $\tau_\ast = +\infty$, there is nothing to prove, and thus we only consider the case $\tau_\ast < +\infty$. Let $(\gamma_{n_k})_{k\in \mathbb N}$ be a subsequence such that $\lim_{k\to \infty} \tau(\gamma_{n_k}) = \tau_\ast$. Since $\Xi$ is closed, we have, from the definition of $\tau$, that $\gamma_{n_k}(\tau(\gamma_{n_k})) \in \Xi$ for every $k \in \mathbb N$, and thus $\gamma(\tau_\ast) \in \Xi$. Hence, by the definition of $\tau$, we deduce that $\tau(\gamma) \leq \tau_\ast$, which is the desired inequality.

Regarding the function $\Psi \circ \tau$, since $\Psi$ is lower semicontinuous and nondecreasing, we obtain that, for any sequence $(\gamma_n)_{n\in \mathbb N}$ in $\Gamma$ converging to some $\gamma \in \Gamma$, we have
\[
\liminf_{n \to \infty} \Psi(\tau(\gamma_n)) \ge \Psi(\liminf_{n \to \infty} \tau(\gamma_n)) \ge \Psi(\tau(\gamma)),
\]
as required.
\end{proof}

We now provide the following technical result, which is a preliminary step towards showing the assertion on $L$ from \ref{Hypo-LH} as well as \ref{Hypo-compact}.

\begin{lemma}
\label{lemm:compactness}
Assume that \ref{Hypo3-Omega}--\ref{l is superlinear} and \ref{Psi1} are satisfied and let $\theta > 1$ be the constant from \ref{l is superlinear}\ref{item:lower-bound-ell} and $L$ be the function defined in Remark~\ref{remk:concrete}. Let $(\gamma_n)_{n \in \mathbb N}$ be a sequence in $\Gamma$ such that the sequence $(L(\gamma_n))_{n \in \mathbb N}$ is bounded. Then, for every $T > 0$, $(\gamma_n)_{n \in \mathbb N}$ is a bounded sequence in $W^{1, \theta}([0, T], \overline\Omega)$ and there exists a subsequence $(\gamma_{n_k})_{k \in \mathbb N}$ of $(\gamma_n)_{n \in \mathbb N}$ and an element $\gamma \in W^{1, \theta}_{\loc}(\mathbb R_+, \overline\Omega)$ such that $\gamma_{n_k} \to \gamma$ as $k \to +\infty$ in the topology of $\Gamma$ and, for every $T > 0$, we have the weak convergence $\dot\gamma_{n_k} \rightharpoonup \dot\gamma$ as $k \to +\infty$ in $L^\theta([0, T], \mathbb R^d)$.
\end{lemma}

\begin{proof}
Let $\kappa > 0$ and $(\gamma_n)_{n \in \mathbb N}$ be a sequence in $\Gamma$ such that $L(\gamma_n) \leq \kappa$ for every $n \in \mathbb N$. Hence, by \ref{l is superlinear}, we deduce that, for every $n \in \mathbb N$,
\[
\int_0^{+\infty} \abs{\dot\gamma_n(t)}^\theta \diff t \leq \kappa.
\]
Thus, for every $T > 0$, the sequence $(\gamma_n)_{n \in \mathbb N}$ is bounded in $W^{1, \theta}([0, T], \overline\Omega)$. Recall also that the injection of $W^{1, \theta}([0, T], \overline\Omega)$ into $\mathcal C([0, T], \overline\Omega)$ is compact
and that bounded sets of $W^{1, \theta}([0, T], \overline\Omega)$ are relatively compact for the weak convergence.

We extract the required subsequence by a standard diagonal argument as follows. Since $(\gamma_n)_{n \in \mathbb N}$ is bounded in $W^{1, \theta}([0, 1], \overline\Omega)$, we extract a subsequence $(\gamma^1_n)_{n \in \mathbb N}$ of $(\gamma_n)_{n \in \mathbb N}$ which converges strongly in $\mathcal C([0, 1], \overline\Omega)$ and weakly in $W^{1, \theta}([0, 1], \overline\Omega)$ to some $\gamma^1 \in W^{1, \theta}([0, 1], \overline\Omega)$. Now, assuming that $k \in \mathbb N^\ast$ is such that we have constructed a subsequence $(\gamma^k_n)_{n \in \mathbb N}$ of $(\gamma_n)_{n \in \mathbb N}$ converging strongly in $\mathcal C([0, k], \overline\Omega)$ and weakly in $W^{1, \theta}([0, k], \overline\Omega)$ to some element $\gamma^k \in W^{1, \theta}([0, k], \overline\Omega)$, using the fact that $(\gamma^k_n)_{n \in \mathbb N}$ is bounded in $W^{1, \theta}([0, k+1], \overline\Omega)$, we extract a subsequence $(\gamma^{k+1}_n)_{n \in \mathbb N}$ of $(\gamma^k_n)_{n \in \mathbb N}$ converging strongly in $\mathcal C([0, k+1], \overline\Omega)$ and weakly in $W^{1, \theta}([0, k+1], \overline\Omega)$ to some $\gamma^{k+1} \in W^{1, \theta}([0, k+1], \overline\Omega)$, and clearly, by uniqueness of the limit, $\gamma^{k+1}$ coincides with $\gamma^k$ in their common domain $[0, k]$.

Define $\gamma\colon \mathbb R_+ \to \overline\Omega$ by setting $\gamma(t) = \gamma^{\ceil{t}}(t)$, and note that, by construction, we have $\gamma \in W^{1, \theta}_{\loc}(\mathbb R_+, \overline\Omega)$, since, for every $k \in \mathbb N^\ast$, $\gamma$ coincides with $\gamma^k$ in $[0, k]$. The diagonal sequence $(\gamma_n^n)_{n \in \mathbb N}$ is a subsequence of $(\gamma_n)_{n \in \mathbb N}$ which, by construction, converges strongly to $\gamma$ in $\Gamma$, i.e., uniformly in any compact subset of $\mathbb R_+$, and, for every $T > 0$, it also converges weakly to $\gamma$ in $W^{1, \theta}([0, T], \overline\Omega)$, yielding the conclusion.
\end{proof}

A first consequence of Lemma~\ref{lemm:compactness} is the following.

\begin{corollary}
\label{coro:L-lsc}
Assume that \ref{Hypo3-Omega}--\ref{l is superlinear} and \ref{Psi1} are satisfied and let $L$ be the function defined in Remark~\ref{remk:concrete}. Then $L$ is lower semicontinuous.
\end{corollary}

\begin{proof}
Let $(\gamma_n)_{n \in \mathbb N}$ be a sequence in $\Gamma$ such that $\gamma_n \to \gamma$ as $n \to +\infty$ for some $\gamma \in \Gamma$. We need to prove that
\begin{equation}
\label{eq:L-lsc}
L(\gamma) \leq \liminf_{n \to +\infty} L(\gamma_n).
\end{equation}
If the right-hand side of \eqref{eq:L-lsc} is equal to $+\infty$, there is nothing to prove, so we only consider the case where it is finite. In this case, we extract from $(\gamma_n)_{n \in \mathbb N}$ a subsequence $(\gamma_{n_k})_{k \in \mathbb N}$ such that $(L(\gamma_{n_k}))_{k \in \mathbb N}$ converges to the right-hand side of \eqref{eq:L-lsc}. For simplicity, we will denote $(\gamma_{n_k})_{k \in \mathbb N}$ by $(\gamma_n)_{n \in \mathbb N}$ in the sequel. Up to removing at most finitely many elements of the sequence, we have $L(\gamma_n) < +\infty$ for every $n \in \mathbb N$, and thus $(L(\gamma_n))_{n \in \mathbb N}$ is bounded. Hence, by Lemma~\ref{lemm:compactness}, up to extracting once again a subsequence, which we still denote by the same notation, we deduce that $\gamma \in W^{1, \theta}_{\loc}(\mathbb R_+, \overline\Omega)$ and that, for every $T > 0$, we have the weak convergence $\dot\gamma_n \rightharpoonup \dot\gamma$ in $L^\theta([0, T], \mathbb R^d)$.

By Lemma~\ref{lemm:tau-psi-lsc}, we have $\Psi(\tau(\gamma)) \leq \liminf_{n \to +\infty} \Psi(\tau(\gamma_n))$. In addition, using \cite[Theorem~4.5]{Giusti2003Direct}, we obtain that, for every $T > 0$,
\[
\int_0^T \ell(t, \gamma(t), \dot\gamma(t)) \diff t \leq \liminf_{n \to +\infty} \int_0^T \ell(t, \gamma_n(t), \dot\gamma_n(t)) \diff t \leq \liminf_{n \to +\infty} \int_0^{+\infty} \ell(t, \gamma_n(t), \dot\gamma_n(t)) \diff t,
\]
and, since this holds for every $T > 0$, we deduce that
\[
\int_0^{+\infty} \ell(t, \gamma(t), \dot\gamma(t)) \diff t \leq \liminf_{n \to +\infty} \int_0^{+\infty} \ell(t, \gamma_n(t), \dot\gamma_n(t)) \diff t,
\]
yielding the conclusion.
\end{proof}

We also deduce from Lemma~\ref{lemm:compactness} and Corollary~\ref{coro:L-lsc} the following result.

\begin{corollary}
\label{coro:compactness}
Assume that \ref{Hypo3-Omega}--\ref{l is superlinear} and \ref{Psi1} are satisfied and let $L$ be the function defined in Remark~\ref{remk:concrete}. Then \ref{Hypo-compact} is satisfied.
\end{corollary}

\begin{proof}
Let $\kappa > 0$ and consider a sequence $(\gamma_n)_{n \in \mathbb N}$ in $\Gamma$ such that $L(\gamma_n) \leq \kappa$ for every $n \in \mathbb N$. Then, by Lemma~\ref{lemm:compactness}, $(\gamma_n)_{n \in \mathbb N}$ admits a subsequence, still denoted by $(\gamma_n)_{n \in \mathbb N}$, converging to some $\gamma \in W^{1, \theta}_{\loc}(\mathbb R_+, \overline\Omega)$, both in the topology of $\Gamma$ and with $\dot\gamma_n \rightharpoonup \dot\gamma$ in $L^\theta([0, T], \mathbb R^d)$ for every $T > 0$. Since $L$ is lower semicontinuous by Corollary~\ref{coro:L-lsc} and $\gamma_n \to \gamma$ in $\Gamma$ as $n \to +\infty$, we deduce that $L(\gamma) \leq \liminf_{n \to +\infty} L(\gamma_n) \leq \kappa$, yielding the conclusion.
\end{proof}

Let us now show that $J$ is also lower semicontinuous and $H$ is Borel measurable.

\begin{lemma}
\label{lemm:J-lsc}
Assume that \ref{Hypo3-Omega}--\ref{Psi1} are satisfied, let $L$ and $H$ be the functions defined in Remark~\ref{remk:concrete}, and $J$ be the function defined in \eqref{eq:def-J}. Then $J$ is lower semicontinuous and $H$ is Borel measurable.
\end{lemma}

\begin{proof}
We note first that it suffices to prove the assertion on $J$. Indeed, since $L$ is lower semicontinuous by Corollary~\ref{coro:L-lsc}, it is Borel measurable, and thus the set $\{(\gamma, \widetilde\gamma) \in \Gamma \times \Gamma \suchthat L(\gamma) = +\infty \text{ or } L(\widetilde\gamma) = +\infty\}$ is Borel measurable. Since $H$ is constant and equal to $+\infty$ in this set, it suffices to show that $H$ is Borel measurable in the complementary of this set, that is, on $\{(\gamma, \widetilde\gamma) \in \Gamma \times \Gamma \suchthat L(\gamma) < +\infty \text{ and } L(\widetilde\gamma) < +\infty\}$. In this set, we have $H(\gamma, \widetilde\gamma) = J(\gamma, \widetilde\gamma) - L(\gamma) - L(\widetilde\gamma)$, and thus, if $J$ is shown to be lower semicontinuous, it will also be Borel measurable, and hence $H$ will be Borel measurable as the difference of Borel measurable functions.

Let us then prove that $J$ is lower semicontinuous. Let $(\gamma_n, \widetilde\gamma_n)_{n \in \mathbb N}$ be a sequence in $\Gamma \times \Gamma$ converging to some $(\gamma, \widetilde\gamma) \in \Gamma \times \Gamma$. We want to prove that
\begin{equation}
\label{eq:to-prove-J}
J(\gamma, \widetilde\gamma) \leq \liminf_{n \to +\infty} J(\gamma_n, \widetilde\gamma_n).
\end{equation}
As in the proof of Corollary~\ref{coro:L-lsc}, we consider only the case where the right-hand side of the above inequality is finite and, up to extracting a subsequence, we also have that $J(\gamma_n, \widetilde\gamma_n)$ converges as $n \to +\infty$ to $\liminf_{n \to +\infty} J(\gamma_n, \widetilde\gamma_n)$ and that the sequence $(J(\gamma_n, \widetilde\gamma_n))_{n \in \mathbb N}$ is bounded. In particular, $(L(\gamma_n))_{n \in \mathbb N}$ and $(L(\widetilde\gamma_n))_{n \in \mathbb N}$ are bounded sequences, so, by Lemma~\ref{lemm:compactness}, we deduce that $\gamma$ and $\widetilde\gamma$ belong to $W^{1, \theta}_{\loc}(\mathbb R_+, \overline\Omega)$ and that, up to a further subsequence extraction, $\dot\gamma_n \rightharpoonup \dot\gamma$ and $\dot{\widetilde\gamma}_n \rightharpoonup \dot{\widetilde\gamma}$ as $n \to +\infty$ in $L^\theta([0, T], \mathbb R^d)$, for every $T > 0$.

Note that, by Corollary~\ref{coro:L-lsc}, we already have
\[
L(\gamma) \leq \liminf_{n \to +\infty} L(\gamma_n), \qquad L(\widetilde\gamma) \leq \liminf_{n \to +\infty} L(\widetilde\gamma_n),
\]
so we are left to show that
\begin{equation}
\label{eq:to-prove-H}
H(\gamma, \widetilde\gamma) \leq \liminf_{n \to +\infty} H(\gamma_n, \widetilde\gamma_n).
\end{equation}

Let us denote $\sigma_n = \tau(\gamma_n) \wedge \tau(\widetilde\gamma_n)$ and $\sigma = \tau(\gamma) \wedge \tau(\widetilde\gamma)$ and recall that, by Lemma~\ref{lemm:tau-psi-lsc}, we have $\sigma \leq \liminf_{n \to +\infty} \sigma_n$. We have nothing to prove in the case $\sigma = 0$ since, in this case, the left-hand side of \eqref{eq:to-prove-H} is zero. We thus only consider the case $\sigma > 0$ from now on. Fix an increasing sequence $(T_k)_{k \in \mathbb N}$ in $\mathbb R_+$ with $\lim_{k \to +\infty} T_k = \sigma$. For each $k \in \mathbb N$, applying \cite[Theorem~4.5]{Giusti2003Direct} to the interval $[0, T_k]$, we obtain that
\begin{equation}
\label{eq:tricky}
\int_0^{T_k} h(t, \gamma(t), \widetilde\gamma(t), \dot\gamma(t), \dot{\widetilde\gamma}(t)) \diff t \leq \liminf_{n \to +\infty} \int_0^{T_k} h(t, \gamma_n(t), \widetilde\gamma_n(t), \dot\gamma_n(t), \dot{\widetilde\gamma}_n(t)) \diff t.
\end{equation}
Since $T_k < \sigma \leq \liminf_{n \to +\infty} \sigma_n$, we have $\sigma_n > T_k$ for $n$ large enough (depending on $k$), and thus, as $h$ is nonnegative, we get
\[
\int_0^{T_k} h(t, \gamma_n(t), \widetilde\gamma_n(t), \dot\gamma_n(t), \dot{\widetilde\gamma}_n(t)) \diff t \leq \int_0^{\sigma_n} h(t, \gamma_n(t), \widetilde\gamma_n(t), \dot\gamma_n(t), \dot{\widetilde\gamma}_n(t)) \diff t
\]
for $n$ large enough. Hence, taking the $\liminf$ as $n \to +\infty$, we obtain, combining with \eqref{eq:tricky}, that
\[
\int_0^{T_k} h(t, \gamma(t), \widetilde\gamma(t), \dot\gamma(t), \dot{\widetilde\gamma}(t)) \diff t \leq \liminf_{n \to +\infty} \int_0^{\sigma_n} h(t, \gamma_n(t), \widetilde\gamma_n(t), \dot\gamma_n(t), \dot{\widetilde\gamma}_n(t)) \diff t.
\]
Since this holds true for every $k \in \mathbb N$, we then obtain that
\[
\int_0^{\sigma} h(t, \gamma(t), \widetilde\gamma(t), \dot\gamma(t), \dot{\widetilde\gamma}(t)) \diff t \leq \liminf_{n \to +\infty} \int_0^{\sigma_n} h(t, \gamma_n(t), \widetilde\gamma_n(t), \dot\gamma_n(t), \dot{\widetilde\gamma}_n(t)) \diff t,
\]
as required.
\end{proof}

\begin{remark}
The proof of Lemma~\ref{lemm:J-lsc} does not imply that $H$ is lower semicontinuous. Indeed, although we prove the inequality \eqref{eq:to-prove-H}, this is not done for any sequence $(\gamma_n, \widetilde\gamma_n)_{n \in \mathbb N}$ converging in $\Gamma \times \Gamma$: as we want to prove \eqref{eq:to-prove-J}, we restrict our attention to sequences for which the right-hand side of \eqref{eq:to-prove-J} is finite. Hence, our proof of the inequality \eqref{eq:to-prove-H} does not take into account sequences $(\gamma_n, \widetilde\gamma_n)_{n \in \mathbb N}$ for which the right-hand side of \eqref{eq:to-prove-H} is finite, but that of \eqref{eq:to-prove-J} is infinite. One should also take these sequences into account in order to obtain lower semicontinuity of $H$, and the main issue is that we cannot apply the compactness result from Lemma~\ref{lemm:compactness} to such sequences.
\end{remark}

Our next result establishes \ref{Hypo-H-leq-L} as a consequence of \ref{Hypo3-Omega}--\ref{Psi1}.

\begin{lemma}
\label{lemm:proof-of-H-leq-L}
Assume that \ref{Hypo3-Omega}--\ref{Psi1} are satisfied and let $L$ and $H$ be the functions defined in Remark~\ref{remk:concrete}. Then \ref{Hypo-H-leq-L} is satisfied.
\end{lemma}

\begin{proof}
Consider the constants $\alpha$, $\theta$, $C$, $\beta$, $a$, and $b$ from \ref{l is superlinear}--\ref{Psi1}. For every $\gamma \in \Gamma$ with $L(\gamma) < +\infty$, we have, using Young's inequality, that
\begin{align*}
\int_0^{\tau(\gamma)} \abs{\dot\gamma(t)}^\beta \diff t & \leq \frac{\beta}{\theta} \int_0^{+\infty} \abs{\dot\gamma(t)}^\theta \diff t + \frac{\theta - \beta}{\theta} \tau(\gamma) \\
& \leq \frac{\beta}{\alpha\theta} \int_0^{+\infty} \ell(t, \gamma(t), \dot\gamma(t)) \diff t + \frac{(\theta - \beta)(\Psi(\tau(\gamma)) + b)}{a \theta} \\
& \leq M \left(L(\gamma) + \frac{1}{2}\right),
\end{align*}
where $M = \max\left\{\frac{\beta}{\alpha\theta}, \frac{\theta - \beta}{a \theta}, \frac{2 b (\theta - \beta)}{a \theta}\right\}$. On the other hand, using \ref{h convex and symmetric}, for every $(\gamma, \widetilde\gamma) \in \Gamma \times \Gamma$ with $L(\gamma) < +\infty$ and $L(\widetilde\gamma) < +\infty$, we have
\[
\begin{aligned}
H(\gamma, \widetilde \gamma) & \leq \int_{0}^{\tau(\gamma) \wedge \tau(\widetilde \gamma)} C (\abs{\dot{\gamma}(t)}^{\beta} +\abs{\dot{\widetilde \gamma}(t)}^{\beta})\diff t \\
&\leq C \left(\int_{0}^{\tau(\gamma)} \abs{\dot \gamma(t)}^{\beta} \diff t \right) + C  \left(\int_{0}^{\tau(\widetilde \gamma)} \abs{\dot{\widetilde \gamma}(t)}^{\beta} \diff t\right) \\
& \leq C M (L(\gamma) + L(\widetilde\gamma) + 1),
\end{aligned}
\]
yielding the conclusion.
\end{proof}

We are thus finally in position to conclude that the mean field game from Section~\ref{Cucker-Smale:The Model} can be studied through the non-atomic game $\NAGall$.

\begin{proposition}
\label{prop:same-equilibria}
Assume that \ref{Hypo3-Omega}--\ref{individualCost} are satisfied. Consider the non-atomic game $\NAGall$ with $X$, $Y$, $\pi$, $L$, and $H$ defined as in Remark~\ref{remk:concrete}. Then \ref{Hypo-XY-Polish}--\ref{Hypo-H-leq-L} are satisfied for $\NAGall$. In addition, $Q$ is an equilibrium of the mean field game from Section~\ref{Cucker-Smale:The Model} with initial condition $m_0 \in \mathcal P(\overline\Omega)$ if and only if it is an equilibrium of $\NAGall$.
\end{proposition}

\begin{proof}
The fact that \ref{Hypo-XY-Polish}--\ref{Hypo-H-leq-L} are satisfied was already established in the previous results. As for the second part of the statement, note that there is a subtlety to be addressed: the term $H(\gamma, \widetilde\gamma)$ does not necessarily coincide with the integral $\int_{0}^{\tau(\gamma)\wedge \tau(\widetilde \gamma)} h(t, \gamma(t),\allowbreak \widetilde{\gamma}(t),\allowbreak \dot{\gamma}(t),\allowbreak \dot{\widetilde{\gamma}}(t)) \diff t$, since the first one is defined to be $+\infty$ whenever $L(\gamma) = +\infty$ or $L(\widetilde\gamma) = +\infty$, but the second one can be finite even in some cases where $L(\gamma) = +\infty$ or $L(\widetilde\gamma) = +\infty$. Hence, the functions defined in \eqref{eq:cost} and \eqref{eq:abstract-cost} do not necessarily coincide for all $(\gamma, Q)$. In this proof, we denote by $F_1(\gamma, Q)$ the function defined in \eqref{eq:cost} and by $F_2(\gamma, Q)$ the function defined in \eqref{eq:abstract-cost}.

Notice first that, if $Q \in \mathcal P(\Gamma)$ is such that $Q(\dom L) = 1$, then $F_1(\gamma, Q) = F_2(\gamma, Q)$ for every $\gamma \in \Gamma$. Indeed, recall that, for every $(\gamma, \widetilde\gamma) \in \dom L \times \dom L$, we have
\[
H(\gamma, \widetilde\gamma) = \int_{0}^{\tau(\gamma)\wedge \tau(\widetilde \gamma)} h(t, \gamma(t), \widetilde{\gamma}(t), \dot{\gamma}(t), \dot{\widetilde{\gamma}}(t)) \diff t.
\]
Since $Q(\dom L) = 1$, we can integrate the above equality in $\widetilde\gamma$ with respect to $Q$ to deduce that
\[
\int_{\Gamma} H(\gamma, \widetilde\gamma) \diff Q(\widetilde\gamma) = \int_{\Gamma} \int_{0}^{\tau(\gamma)\wedge \tau(\widetilde \gamma)} h(t, \gamma(t), \widetilde{\gamma}(t), \dot{\gamma}(t), \dot{\widetilde{\gamma}}(t)) \diff t \diff Q(\widetilde\gamma)
\]
for every $\gamma \in \dom L$. Thus, $F_1(\gamma, Q) = F_2(\gamma, Q)$ for every $\gamma \in \dom L$. On the other hand, if $\gamma \notin \dom L$, we have $L(\gamma) = +\infty$, hence $F_1(\gamma, Q) = F_2(\gamma, Q) = +\infty$. Hence, $F_1(\gamma, Q) = F_2(\gamma, Q)$ for every $\gamma \in \Gamma$.

If $Q$ is an equilibrium of the mean field game from Section~\ref{Cucker-Smale:The Model} with initial condition $m_0 \in \mathcal P(\overline\Omega)$, then $\int_{\Gamma} F_1 (\gamma, Q) \diff Q(\gamma) < +\infty$, showing that $F_1(\gamma, Q) < +\infty$ for $Q$-almost every $\gamma \in \Gamma$. In particular, $L(\gamma) < +\infty$ for $Q$-almost every $\gamma$, i.e., $Q(\dom L) = 1$, and thus, by the above argument, $F_1(\gamma, Q) = F_2(\gamma, Q)$ for every $\gamma \in \Gamma$, proving that $Q$ is also an equilibrium of $\NAGall$. Conversely, if $Q$ is an equilibrium of $\NAGall$, then $\int_{\Gamma} F_2 (\gamma, Q) \diff Q(\gamma) < +\infty$, showing that $F_2(\gamma, Q) < +\infty$ for $Q$-almost every $\gamma \in \Gamma$. In particular, $L(\gamma) < +\infty$ for $Q$-almost every $\gamma$, i.e., $Q(\dom L) = 1$, and thus, by the above argument, $F_1(\gamma, Q) = F_2(\gamma, Q)$ for every $\gamma \in \Gamma$, proving that $Q$ is also an equilibrium of the mean field game from Section~\ref{Cucker-Smale:The Model} with initial condition $m_0 \in \mathcal P(\overline\Omega)$.
\end{proof}

As an immediate consequence of Theorem~\ref{thm:exist-abstract} and Proposition~\ref{prop:same-equilibria}, we obtain our main result on the mean field game model from Section~\ref{Cucker-Smale:The Model}.

\begin{theorem}
\label{thm:3:exist}
Assume that \ref{Hypo3-Omega}--\ref{individualCost} are satisfied and let $m_0 \in \mathcal P(\overline\Omega)$. Then there exists an equilibrium $Q \in \mathcal P(\Gamma)$ of the mean field game from Section~\ref{Cucker-Smale:The Model} with initial condition $m_0$.
\end{theorem}

We conclude our discussion on the model from Section~\ref{Cucker-Smale:The Model} by addressing strong equilibria. Similarly to Definition~\ref{def:strong-equilibrium-potential}, we shall say that $Q \in \mathcal P(\Gamma)$ is a strong equilibrium of the mean field game from Section~\ref{Cucker-Smale:The Model} if it satisfies Definition~\ref{def:equilibrium-3} with ``$Q$-almost every $\gamma$'' replaced by ``every $\gamma \in \supp(Q)$''.

Let us introduce two additional assumptions on the model from Section~\ref{Cucker-Smale:The Model}.

\begin{hypothesisMFG}[resume]
\item\label{Hypo-geodesic} For every $x \in \overline\Omega$ and every sequence $(x_n)_{n \in \mathbb N}$ with $x_n \to x$ as $n \to +\infty$, we have that $\dist_{\mathrm{geo}}(x_n, x) \to 0$ as $n \to +\infty$, where $\dist_{\mathrm{geo}}$ denotes the geodesic distance in $\overline\Omega$.

\item\label{Hypo-ultima} There exists $\alpha^\ast > 0$ such that $\ell(t, x, p) \leq \alpha^\ast \abs{p}^\theta$ for every $(t, x, p) \in \mathbb R_+ \times \overline\Omega \times \mathbb R^d$, where $\theta$ is the same constant as in \ref{l is superlinear}.
\end{hypothesisMFG}

One of the consequences of \ref{Hypo-ultima} is that trajectories minimizing $F(\gamma, Q)$ starting at some $x_0$ already in the target set $\Xi$ must necessarily remain constant at all times.

\begin{lemma}
\label{lemm:optimal-constant}
Assume that \ref{Hypo3-Omega}--\ref{Psi1} and \ref{Hypo-ultima} are satisfied for the mean field game of Section~\ref{Cucker-Smale:The Model} and let $x_0 \in \Xi$ and $Q \in \mathcal P(\Gamma)$. Then the unique trajectory $\gamma_0$ minimizing \eqref{eq:cost} with the constraint $\gamma_0(0) = x_0$ is the constant trajectory $\gamma_0(t) = x_0$ for every $t \in \mathbb R_+$.
\end{lemma}

\begin{proof}
Clearly, $F(\gamma, Q) \geq \Psi(0)$ for every $\gamma \in \Gamma$, and the value $\Psi(0)$ is attained for the constant trajectory $\gamma_0$ since $\tau(\gamma_0) = 0$ and $0 \leq \ell(t, \gamma_0(t), \dot\gamma_0(t)) \leq \alpha^\ast \abs{\dot\gamma_0(t)}^\theta = 0$. To prove uniqueness, note that $F(\gamma, Q) = \Psi(0)$ implies $\int_0^{+\infty} \ell(t, \gamma(t), \dot\gamma(t)) \diff t = 0$, showing that $\gamma$ is absolutely continuous and $\ell(t, \gamma(t), \dot\gamma(t)) = 0$ for almost every $t \in \mathbb R_+$. By \ref{l is superlinear}\ref{item:lower-bound-ell}, this implies that $\dot\gamma(t) = 0$ for almost every $t \in \mathbb R_+$, and hence $\gamma$ is constant. The constraint $\gamma(0) = x_0$ then implies that $\gamma = \gamma_0$.
\end{proof}

We now use \ref{Hypo-geodesic} and \ref{Hypo-ultima} to prove that \ref{Hypo-technical} is satisfied.

\begin{lemma}
\label{lemm:OOpt-closed}
Assume that \ref{Hypo3-Omega}--\ref{individualCost}, \ref{Hypo-geodesic}, and \ref{Hypo-ultima} are satisfied for the mean field game of Section~\ref{Cucker-Smale:The Model} and consider the non-atomic game $\NAGall$ defined as in Remark~\ref{remk:concrete}. Then \ref{Hypo-technical} is satisfied.
\end{lemma}

\begin{proof}
Let $Q \in \dom \mathcal L$ and $\OOpt(Q)$ be the set defined in \eqref{eq:def-OOpt}. Recall that we have the inclusion
\[\OOpt(Q_0) \subset \{\gamma \in \Gamma \suchthat L(\gamma) \leq (C+1)\kappa + C \mathcal L(Q_0) + C\},\]
where $\kappa$ and $C$ are the constants from \ref{Hypo-domain} and \ref{Hypo-H-leq-L}: this was proved in Remark~\ref{remk:new-proof} under \ref{Hypo-domain} and \ref{Hypo-H-leq-L}, and these two assumptions are satisfied here thanks to \ref{individualCost} and Lemma~\ref{lemm:proof-of-H-leq-L}. Let $(\gamma_n)_{n \in \mathbb N}$ be a sequence in $\OOpt(Q)$ and $\gamma \in \Gamma$ with $\gamma_n \to \gamma$ as $n \to +\infty$. In particular, by Lemma~\ref{lemm:compactness}, we deduce that $\gamma$ and $\gamma_n$ belong to $W^{1, \theta}_{\loc}(\mathbb R_+, \overline\Omega)$ for every $n \in \mathbb N$ and that $\dot{\gamma}_n \rightharpoonup \dot\gamma$ in $L^\theta([0, T], \mathbb R^d)$ as $n \to +\infty$, for every $T > 0$.

Let $x_0 = \gamma(0)$ and, for $n \in \mathbb N$, set $x_n = \gamma_n(0)$. We split the proof in three cases.

\case{1}{$x_0 \in \Xi$ and there exists a subsequence $(x_{n_k})_{k \in \mathbb N}$ of $(x_n)_{n \in \mathbb N}$ such that $x_{n_k} \in \Xi$ for every $k \in \mathbb N$}

In this case, it follows from Lemma~\ref{lemm:optimal-constant} that $\gamma_{n_k}$ is constant and, since $\gamma_{n_k} \to \gamma$ as $k \to +\infty$, it follows that $\gamma$ is also constant. Hence, applying once again Lemma~\ref{lemm:optimal-constant}, we deduce that $\gamma$ minimizes \eqref{eq:cost} with the constraint $\gamma(0) = x_0$, yielding that $\gamma \in \OOpt(Q)$.

\case{2}{$x_0 \in \Xi$ and $x_{n} \notin \Xi$ for every $n \in \mathbb N$}

Let $\varepsilon_n = \dist_{\mathrm{geo}}(x_n, x_0)$ and $\sigma_{n}\colon[0, \varepsilon_n]\to \overline\Omega$ be a geodesic curve such that $\sigma_{n}(0) = x_{n}$, $\sigma_{n}(\varepsilon_n) = x_0$, and $\abs*{\dot \sigma_{n}(t)} = 1$ for almost every $t\in (0,\varepsilon_n)$. Notice that the existence of such a geodesic curve can be ensured by assumption~\ref{Hypo-geodesic} and \cite[Proposition~2.5.19]{Burago2001Course} and, in addition, by \ref{Hypo-geodesic}, we have $\varepsilon_n \to 0$ as $n\to \infty$. One can easily extend the domain of $\sigma_n$ to $\mathbb R_+$ by setting $\sigma_n(t) = x_0$ for all $t\ge \varepsilon_n$. We have $0 < \tau(\sigma_n) \leq \varepsilon_n \to 0$ as $n \to +\infty$, and thus $\Psi(\tau(\sigma_n)) \to \Psi(0^+)$ as $n \to +\infty$, where $\Psi(0^+) = \lim_{t \to 0^+} \Psi(t)$, which exists since $\Psi$ is nondecreasing. We also compute
\begin{equation}
\label{eq:convergence-ell}
\int_0^{+\infty} \ell(t, \sigma_n(t), \dot\sigma_n(t)) \diff t \leq \alpha^\ast \int_0^{+\infty} \abs{\dot\sigma_n(t)}^\theta \diff t \leq \alpha^\ast \varepsilon_n \xrightarrow[n \to +\infty]{} 0
\end{equation}
and, for every $\omega \in \dom L$,
\begin{equation}
\label{eq:pointwise-convergence}
\int_{0}^{\tau(\sigma_n) \wedge \tau(\omega)} h(t, \sigma_n(t), \omega(t), \dot\sigma_n(t), \dot{\omega}(t)) \diff t \leq C \varepsilon_n + C \int_0^{\varepsilon_n} \abs{\dot{\omega}(t)}^\beta \diff t \xrightarrow[n \to +\infty]{} 0,
\end{equation}
where we use the fact that $\omega \in \dom L$ to deduce from \ref{l is superlinear}\ref{item:lower-bound-ell} that $\dot{\omega} \in L^\theta(\mathbb R_+, \mathbb R^d)$, and thus $\abs{\dot{\omega}}^\beta \in L^1([0, T], \mathbb R^d)$ for every $T > 0$ since $\beta \in (0, \theta]$. In addition, for all $n$ large enough so that $\varepsilon_n \leq 1$, we have
\begin{equation}
\label{eq:domination}
\int_{0}^{\tau(\sigma_n) \wedge \tau(\omega)} h(t, \sigma_n(t), \omega(t), \dot\sigma_n(t), \dot{\omega}(t)) \diff t \leq C + C \int_0^1 \abs{\dot{\omega}(t)}^\beta \diff t,
\end{equation}
and the right-hand side of the above inequality is $Q$-integrable, since, by Hölder's inequality and \ref{l is superlinear}\ref{item:lower-bound-ell}, we have
\begin{equation}
\label{eq:domination-2}
\begin{aligned}
\int_{\Gamma} \int_0^1 \abs{\dot{\omega}(t)}^\beta \diff t \diff Q(\omega) & \leq \left(\int_{\Gamma} \int_0^1 \abs{\dot{\omega}(t)}^\theta \diff t \diff Q(\omega)\right)^{\frac{\beta}{\theta}} \\
& \leq \left(\frac{1}{\alpha} \int_{\Gamma} \int_0^{+\infty} \ell(t, \omega(t), \dot{\omega}(t)) \diff t \diff Q(\omega)\right)^{\frac{\beta}{\theta}} \leq \left(\frac{\mathcal L(Q)}{\alpha}\right)^{\frac{\beta}{\theta}} < +\infty.
\end{aligned}
\end{equation}
Hence, it follows from \eqref{eq:pointwise-convergence}, \eqref{eq:domination}, and Lebesgue's dominated convergence theorem that
\[
\lim_{n \to +\infty} \int_{\Gamma}\int_{0}^{\tau(\sigma_n) \wedge \tau(\omega)} h(t, \sigma_n(t), \omega(t), \dot\sigma_n(t), \dot{\omega}(t)) \diff t \diff Q(\omega) = 0.
\]
Combining this with \eqref{eq:convergence-ell} and the fact that $\Psi(\tau(\sigma_n)) \to \Psi(0^+)$ as $n \to +\infty$, we deduce that $\lim_{n \to +\infty} F(\sigma_{n}, Q)  = \Psi(0^+)$. Since $\gamma_n$ is a minimizer of \eqref{eq:cost} with fixed initial condition $x_n$ and $\sigma_n \in \Gamma$ with $\sigma_n(0) = x_n$, we have $F(\gamma_{n}, Q) \leq F(\sigma_{n},Q)$, and, on the other hand, since $\tau(\gamma_n) > 0$, we have $\Psi(0^+) \leq \Psi(\tau(\gamma_n)) \leq F(\gamma_n, Q)$. Hence, $\lim_{n\to +\infty}F(\gamma_{n}, Q) = \Psi(0^+)$, and \ref{l is superlinear}\ref{item:lower-bound-ell} yields $\int_{0}^{+\infty} \abs{\dot\gamma_{n}(t)}^\theta \diff t \to 0$ as $n \to +\infty$. Recalling that, for every $T > 0$, we have $\dot{\gamma}_n \rightharpoonup \dot\gamma$ in $L^\theta([0, T], \mathbb R^d)$ as $n \to +\infty$, we deduce, for the lower semicontinuity of the $L^\theta$ norm with respect to weak convergence, that
\[
\norm{\dot\gamma}_{L^\theta([0, T], \mathbb R^d)} \leq \liminf_{n \to +\infty} \norm{\dot\gamma_n}_{L^\theta([0, T], \mathbb R^d)} = 0,
\]
and, since $T > 0$ is arbitrary, we deduce that $\gamma$ is constant. Hence, by Lemma~\ref{lemm:optimal-constant}, we deduce that $\gamma$ minimizes \eqref{eq:cost} with the constraint $\gamma(0) = x_0$, yielding that $\gamma \in \OOpt(Q)$.

\case{3}{$x_0 \notin \Xi$}

By \ref{Hypo3-Gamma}, we have in this case $x_n \notin \Xi$ for $n$ large enough. To prove that $\gamma \in \OOpt(Q)$, we will prove that $F(\gamma, Q) \leq F(\widetilde\gamma, Q)$ for every $\widetilde\gamma \in \Gamma$ with $\widetilde\gamma(0) = x_0$. For that purpose, we will construct a sequence of trajectories $(\widetilde\gamma_n)_{n \in \mathbb N}$ in $\Gamma$ with $\widetilde\gamma_n(0) = x_n$ for every $n \in \mathbb N$ and such that
\begin{equation}
\label{eq:to-prove-FgammaQ}
\liminf_{n \to +\infty} F(\widetilde\gamma_n, Q) \leq F(\widetilde\gamma, Q).
\end{equation}
This will allow us to conclude, since, by Lemma~\ref{lemm:F-lsc} and by the fact that $\gamma_n \in \OOpt(Q)$, we will then have
\[
F(\gamma, Q) \leq \liminf_{n \to +\infty} F(\gamma_n, Q) \leq \liminf_{n \to +\infty} F(\widetilde\gamma_n, Q) \leq F(\widetilde\gamma, Q),
\]
yielding that $\gamma \in \OOpt(Q)$. We then focus on showing \eqref{eq:to-prove-FgammaQ}.

Let $\widetilde\gamma \in \Gamma$ with $\widetilde\gamma(0) = x_0$. Note that there is nothing to be prove in \eqref{eq:to-prove-FgammaQ} if $F(\widetilde\gamma, Q) = +\infty$, so we assume in the sequel that $F(\widetilde\gamma, Q) < +\infty$. In particular, $\widetilde\gamma \in \dom L$, which implies by \ref{l is superlinear} and \ref{Psi1} that $\dot{\widetilde\gamma} \in L^\theta(\mathbb R_+, \mathbb R^d)$ and $\tau(\widetilde\gamma) < +\infty$.

Let $\varepsilon_n = \dist_{\mathrm{geo}}(x_n, x_0)$ and $\varsigma_n\colon [0, \varepsilon_n]\to \overline \Omega$ be a geodesic curve such that $\varsigma_n(0) = x_n$, $\varsigma_n(\varepsilon_n) = \widetilde\gamma(\varepsilon_n)$, and $\abs*{\dot\varsigma_n(t)} = \frac{\dist_{\mathrm{geo}}(x_n, \widetilde\gamma(\varepsilon_n))}{\varepsilon_n}$ for every $t\in [0, \varepsilon_n]$ (with the convention that, in the case $\varepsilon_n = 0$, we have that $\varsigma_n\colon \{0\} \to \overline\Omega$ is given by $\varsigma_n(0) = x_n = x_0$). Notice that, as in Case~2, the existence of such a geodesic curve can be ensured by assumption~\ref{Hypo-geodesic} and \cite[Proposition~2.5.19]{Burago2001Course} and, in addition, by \ref{Hypo-geodesic}, we have $\varepsilon_n \to 0$ as $n \to +\infty$. Let us define $\widetilde\gamma_n \in \Gamma$ by
\begin{equation*}
\widetilde\gamma_n(t) = \begin{dcases*}
\varsigma_n (t) & if $t \in [0, \varepsilon_n]$, \\
\widetilde\gamma(t) & if $t \geq \varepsilon_n$.
\end{dcases*}
\end{equation*}
In particular, we also have $\widetilde\gamma_n \in \dom L$, $\dot{\widetilde\gamma}_n \in L^\theta(\mathbb R_+, \mathbb R^d)$, and $0 < \tau(\widetilde\gamma_n) \leq \tau(\widetilde\gamma) < +\infty$ for every $n \in \mathbb N$ large enough.

Note that, since $\widetilde\gamma$ and $\widetilde\gamma_n$ belong to $\dom L$, we have
\begin{equation}
\label{eq:estim-ell}
\begin{aligned}
\MoveEqLeft \abs*{\int_0^{+\infty} \ell(t, \widetilde\gamma(t), \dot{\widetilde\gamma}(t)) \diff t - \int_0^{+\infty} \ell(t, \widetilde\gamma_n(t), \dot{\widetilde\gamma}_n(t)) \diff t} \\
& \leq \int_0^{\varepsilon_n} \ell(t, \widetilde\gamma(t), \dot{\widetilde\gamma}(t)) \diff t + \int_0^{\varepsilon_n} \ell(t, \varsigma_n(t), \dot\varsigma_n(t)) \diff t \\
& \leq \int_0^{\varepsilon_n} \ell(t, \widetilde\gamma(t), \dot{\widetilde\gamma}(t)) \diff t + \alpha^\ast \frac{\dist_{\mathrm{geo}}(x_n, \widetilde\gamma(\varepsilon_n))^\theta}{\varepsilon_n^{\theta-1}}.
\end{aligned}
\end{equation}
Since $\widetilde\gamma \in \dom L$, we have
\begin{equation}
\label{eq:estim-ell-1}
\int_0^{\varepsilon_n} \ell(t, \widetilde\gamma(t), \dot{\widetilde\gamma}(t)) \diff t \xrightarrow[n \to +\infty]{} 0.
\end{equation}
Notice also that
\[
\dist_{\mathrm{geo}}(x_n, \widetilde \gamma(\varepsilon_n))^{\theta} \leq \left(\dist_{\mathrm{geo}}(x_n, x_0) + \dist_{\mathrm{geo}}(\widetilde\gamma(0), \widetilde\gamma(\varepsilon_n))\right)^{\theta} \leq 2^{\theta-1} \left(\text{$\dist_{\mathrm{geo}}(x_n, x_0)$}^{\theta} + \left(\int_{0}^{\varepsilon_n} \abs*{\dot{\widetilde \gamma}(t)} \diff t\right)^{\theta}\right)
\]
and, by Hölder's inequality,
\[
\left(\int_{0}^{\varepsilon_n} \abs*{\dot{\widetilde \gamma}(t)} \diff t\right)^{\theta} \leq \varepsilon_n^{\theta - 1} \int_{0}^{\varepsilon_n} \abs*{\dot{\widetilde \gamma}(t)}^{\theta} \diff t.
\]
Hence 
\begin{equation}
\label{Prop-FracGeo}
\frac{\text{$\dist_{\mathrm{geo}}(x_n, \widetilde \gamma(\varepsilon_n))^{\theta}$}}{\varepsilon_n^{\theta -1}} \le 2^{\theta-1} \left( \varepsilon_n + \int_{0}^{\varepsilon_n} \abs*{\dot{\widetilde \gamma}(t)}^{\theta} \diff t \right) \xrightarrow[n \to +\infty]{} 0,
\end{equation}
since $\dot {\widetilde \gamma} \in L^{\theta}(\mathbb R_+, \overline\Omega)$. Combining \eqref{eq:estim-ell}, \eqref{eq:estim-ell-1}, and \eqref{Prop-FracGeo}, we deduce that
\begin{equation}
\label{eq:limit-ell}
\lim_{n \to +\infty} \int_0^{+\infty} \ell(t, \widetilde\gamma_n(t), \dot{\widetilde\gamma}_n(t)) \diff t = \int_0^{+\infty} \ell(t, \widetilde\gamma(t), \dot{\widetilde\gamma}(t)) \diff t.
\end{equation}

Note now that, for every $t \in [0, \varepsilon_n]$, we have
\[
\begin{aligned}
\abs*{\varsigma_n(t) - x_0} & = \abs*{\varsigma_n(0) + \int_{0}^{t} \dot \varsigma_n(s)\diff s -x_0} \le \abs*{x_n - x_0} + t\frac{\dist_{\mathrm{geo}}(x_n,\widetilde \gamma(\varepsilon_n))}{\varepsilon_n}   \\
&\le 
\abs*{x_n - x_0} + \dist_{\mathrm{geo}}(x_n,\widetilde \gamma(\varepsilon_n)) \xrightarrow[n \to +\infty]{} 0,
\end{aligned}
\]
where we use \eqref{Prop-FracGeo} in the last step. Hence, since $x_0 \notin \Xi$, we deduce that, for $n$ large enough (independently of $t$), $\varsigma_n(t) \notin \Xi$ for every $t \in [0, \varepsilon_n]$. In addition, for $n$ large enough (independently of $t$), we clearly have $\widetilde\gamma(t) \notin \Xi$ for every $t \in [0, \varepsilon_n]$. Hence $\tau(\widetilde \gamma_n) = \tau(\widetilde \gamma)$ for $n$ large enough, and thus $\Psi(\tau(\widetilde \gamma_n)) = \Psi(\tau(\widetilde \gamma))$ for $n$ large enough. Combining this with \eqref{eq:limit-ell}, we deduce that
\begin{equation}
\label{eq:limit-L}
\lim_{n \to +\infty} L(\widetilde\gamma_n) = L(\widetilde\gamma).
\end{equation}

Let us now study whether $\int_{\Gamma} H(\widetilde\gamma_n, \omega) \diff Q(\omega)$ converges to $\int_{\Gamma} H(\widetilde\gamma, \omega) \diff Q(\omega)$ as $n \to +\infty$. For every $\omega \in \dom L$, recalling that $\tau(\widetilde\gamma_n) = \tau(\widetilde\gamma)$ for $n$ large enough, we have
\begin{equation}
\label{eq:estim-h}
\abs{H(\widetilde\gamma_n, \omega) - H(\widetilde\gamma, \omega)} \leq \int_0^{\varepsilon_n} h(t, \widetilde\gamma(t), \omega(t), \dot{\widetilde\gamma}(t), \dot\omega(t)) \diff t + \int_0^{\varepsilon_n} h(t, \varsigma_n(t), \omega(t), \dot{\varsigma}_n(t), \dot\omega(t)) \diff t.
\end{equation}
We have
\begin{equation}
\label{eq:estim-h-1}
\int_0^{\varepsilon_n} h(t, \widetilde\gamma(t), \omega(t), \dot{\widetilde\gamma}(t), \dot\omega(t)) \diff t \leq C \left(\int_0^{\varepsilon_n} \abs{\dot{\widetilde\gamma}(t)}^\beta \diff t + \int_0^{\varepsilon_n} \abs{\dot\omega(t)}^\beta \diff t\right) \xrightarrow[n \to +\infty]{} 0
\end{equation}
since $\widetilde\gamma$ and $\omega$ belong to $\dom L$ and, as such, their time derivatives belong to $L^\theta(\mathbb R_+, \mathbb R^d)$, which implies that $\abs{\dot{\widetilde\gamma}}^\beta$ and $\abs{\dot\omega}^\beta$ belong to $L^1([0, T], \mathbb R^d)$ for every $T > 0$, since $\beta \in (0, \theta]$. On the other hand,
\begin{equation}
\label{eq:estim-h-2}
\int_0^{\varepsilon_n} h(t, \varsigma_n(t), \omega(t), \dot{\varsigma}_n(t), \dot\omega(t)) \diff t \leq C \left(\frac{\dist_{\mathrm{geo}}(x_n, \widetilde\gamma(\varepsilon_n))^\beta}{\varepsilon_n^{\beta - 1}} + \int_0^{\varepsilon_n} \abs{\dot\omega(t)}^\beta \diff t\right).
\end{equation}
As before, we have $\int_0^{\varepsilon_n} \abs{\dot\omega(t)}^\beta \diff t \to 0$ as $n \to +\infty$, and, using \eqref{Prop-FracGeo}, we have
\[
\frac{\dist_{\mathrm{geo}}(x_n, \widetilde\gamma(\varepsilon_n))^\beta}{\varepsilon_n^{\beta - 1}} = \varepsilon_n^{\frac{\theta - \beta}{\theta}} \left(\frac{\dist_{\mathrm{geo}}(x_n, \widetilde\gamma(\varepsilon_n))^\theta}{\varepsilon_n^{\theta - 1}}\right)^{\frac{\beta}{\theta}} \xrightarrow[n \to +\infty]{} 0
\]
since $\beta \in (0, \theta]$. Combining the above with \eqref{eq:estim-h}, \eqref{eq:estim-h-1}, and \eqref{eq:estim-h-2}, we deduce that
\begin{equation}
\label{eq:pointwise-convergence-2}
\lim_{n \to +\infty} H(\widetilde\gamma_n, \omega) = H(\widetilde\gamma, \omega)
\end{equation}
for every $\omega \in \dom L$. Notice also that we can estimate, for $n$ large enough (independently of $\omega$) in order to have in particular $\varepsilon_n \leq 1$,
\begin{equation}
\label{eq:domination-3}
\begin{aligned}
H(\widetilde\gamma_n, \omega) & \leq H(\widetilde\gamma, \omega) + \int_0^{\varepsilon_n} h(t, \varsigma_n(t), \omega(t), \dot{\varsigma}_ n(t), \dot\omega(t)) \diff t \\
& \leq H(\widetilde\gamma, \omega) + C \frac{\text{$\dist_{\mathrm{geo}}$}(x_n, \widetilde\gamma(\varepsilon_n))^\beta}{\varepsilon_n^{\beta - 1}} + C \int_0^1 \abs{\dot\omega(t)}^\beta \diff t \\
& \leq H(\widetilde\gamma, \omega) + 1 + C \int_0^1 \abs{\dot\omega(t)}^\beta \diff t,
\end{aligned}
\end{equation}
where we have chosen $n$ large enough to have $C \frac{\text{$\dist_{\mathrm{geo}}$}(x_n, \widetilde\gamma(\varepsilon_n))^\beta}{\varepsilon_n^{\beta - 1}} \leq 1$. The right-hand side of \eqref{eq:domination-3} is $Q$-integrable with respect to $\omega$ thanks to Lemma~\ref{lemm:proof-of-H-leq-L}, the facts that $\widetilde\gamma \in \dom L$ and $Q \in \dom \mathcal L$, and \eqref{eq:domination-2}. Hence, by \eqref{eq:pointwise-convergence-2}, \eqref{eq:domination-3}, and Lebesgue's dominated convergence theorem, we deduce that
\begin{equation}
\label{eq:limit-H}
\lim_{n \to +\infty} \int_{\Gamma} H(\widetilde\gamma_n, \omega) \diff Q(\omega) = \int_{\Gamma} H(\widetilde\gamma, \omega) \diff Q(\omega).
\end{equation}

Combining \eqref{eq:limit-L} and \eqref{eq:limit-H}, we deduce that
\[
\lim_{n \to +\infty} F(\widetilde\gamma_n, Q) = F(\widetilde\gamma, Q)
\]
which clearly implies \eqref{eq:to-prove-FgammaQ}.
\end{proof}

Finally, as a consequence of Theorem~\ref{thm:strong-iff-equilibrium} and Lemma~\ref{lemm:OOpt-closed}, we immediately obtain the following result.

\begin{theorem}
\label{thm:3:equiv}
Assume that \ref{Hypo3-Omega}--\ref{individualCost}, \ref{Hypo-geodesic}, and \ref{Hypo-ultima} are satisfied and let $m_0 \in \mathcal P(\overline\Omega)$. Then $Q \in \mathcal P(\Gamma)$ is an equilibrium of the mean field game from Section~\ref{Cucker-Smale:The Model} with initial condition $m_0$ if and only if it is a strong equilibrium of the same mean field game with the same initial condition.
\end{theorem}

\section{Numerical illustration}
\label{sec:illustration}

We now provide a numerical illustration for the mean field game from Section~\ref{Cucker-Smale:The Model}. In addition to \ref{Hypo3-Omega}--\ref{individualCost}, we work here under the additional assumption that
\[
h(t, x, \widetilde x, p, \widetilde p) = 0 \qquad \text{ if } (x, p) = (\widetilde x, \widetilde p),
\]
which implies that the interaction cost between an agent and themself is zero.

Note that the corresponding $N$-player game is also a potential game. More precisely, consider the game with $N$ agents where the aim of agent $i \in \{1, \dotsc, N\}$ starting at the position $x_{0, i}$ is to minimize with respect to $\gamma_i \in \Gamma$ the cost
\[
F_N(\gamma_i, \gamma_{-i}) = L(\gamma_i) + \frac{1}{N} \sum_{j=1}^N H(\gamma_i, \gamma_j)
\]
with the constraint $\gamma_i(0) = x_{0, i}$, where $L$ and $H$ are given as in Remark~\ref{remk:concrete} and $\gamma_{-i} = (\gamma_1, \dotsc, \gamma_{i-1},\allowbreak \gamma_{i+1}, \dotsc, \gamma_N) \in \Gamma^{N-1}$ represents the choices of trajectories of other agents. Note that, since $H(\gamma_i, \gamma_i) = 0$, the above sum can be equivalently taken over $j \in \{1, \dotsc, N\} \setminus \{i\}$. It is easy to verify that this $N$-player game is a potential game, with potential
\begin{equation}
\label{eq:J_N}
\mathcal J_N(\gamma_1, \dotsc, \gamma_N) = \frac{2}{N} \sum_{i=1}^N L(\gamma_i) + \frac{1}{N^2} \sum_{i=1}^N \sum_{j=1}^N H(\gamma_i, \gamma_j).
\end{equation}
In addition, $F_N(\gamma_i, \gamma_{-i})$ is nothing but the cost $F(\gamma_i, Q)$ from \eqref{eq:cost} with $Q = \frac{1}{N}\sum_{j=1}^N \delta_{\gamma_j}$, and $\mathcal J_N(\gamma_1, \dotsc, \gamma_N)$ is nothing but $\mathcal J(Q)$ with the same $Q$ and $\mathcal J$ given by \eqref{eq:def-mathcal-J}.

Numerical approximations of equilibria of this $N$-player game can be obtained by minimizing the function $\mathcal J_N$ from \eqref{eq:J_N} with the constraints $\gamma_i(0) = x_{0, i}$ for every $i \in \{1, \dotsc, N\}$. In order to avoid direct minimization over the whole space $\Gamma^N$, one can minimize $\mathcal J_N$ one variable at each time through a classical coordinate descent algorithm (see, e.g., \cite{Wright2015Coordinate} for an overview of coordinate descent algorithms in finite dimension), as described in Algorithm~\ref{algo:CoordinateDescent}.

\begin{algorithm}
\caption[Coordinate descent algorithm for minimizing the function $\mathcal J_N$]{Coordinate descent algorithm for minimizing the function $\mathcal J_N$ from \eqref{eq:J_N} with constraints $\gamma_i(0) = x_{0, i}$ for every $i \in \{1, \dotsc, N\}$}
\label{algo:CoordinateDescent}
\begin{algorithmic}
\REQUIRE Positive integer $N$, initial conditions $x_{0, 1}, \dotsc, x_{0, N}$ in $\overline\Omega$.
\STATE Initialize $\gamma_1, \dotsc, \gamma_N$ in $\Gamma$ with $\gamma_i(0) = x_{0, i}$ for every $i \in \{1, \dotsc, N\}$.
\REPEAT
  \FOR{$i \in \{1, \dotsc, N\}$}
	  \STATE Select a new $\gamma_i \in \displaystyle\argmin_{\substack{\tilde\gamma_i \in \Gamma \\ \tilde\gamma_i(0) = x_{0, i}}} \mathcal J_N(\gamma_1, \dotsc, \gamma_{i-1}, \tilde\gamma_i, \gamma_{i+1}, \dotsc, \gamma_N)$
	\ENDFOR
\UNTIL{some convergence criterion is met.}
\end{algorithmic}
\end{algorithm}

Note that, even though $\mathcal J_N$ is a sum over $N + N^2$ terms, at each iteration of the algorithm, when minimizing for some agent $i$, $\mathcal J_N(\gamma_1, \dotsc, \gamma_{i-1}, \widetilde\gamma_i, \gamma_{i+1}, \dotsc, \gamma_N)$ can be replaced by the individual cost $F_N(\widetilde\gamma_i, \gamma_{-i})$, which is a sum of only $1 + N$ terms, as all other terms appearing in $\mathcal J_N(\gamma_1, \dotsc, \gamma_{i-1}, \widetilde\gamma_i, \gamma_{i+1}, \dotsc, \gamma_N)$ are independent of $\widetilde\gamma_i$. Hence, the coordinate descent algorithm coincides with a best response algorithm from game theory, in which, at each step, one player changes their strategy to their best response, while all other players remain with the same strategy (see, e.g., \cite{Monderer1996Potential}, where it is shown that the best response algorithm converges to a Nash equilibrium for potential games with finitely many actions for each player).

For our numerical illustration, we have applied Algorithm~\ref{algo:CoordinateDescent} to a slight modification of the mean field game from Section~\ref{Cucker-Smale:The Model} in which there are two different populations. More precisely, we consider that there are two populations evolving in $\overline\Omega$, the distributions of their trajectories being described by two probability measures $Q_1, Q_2 \in \mathcal P(\Gamma)$. The populations are identical except for their initial distributions and target sets: for $i \in \{1, 2\}$, we assume that population $i$ is distributed at time $t = 0$ according to a measure $m_0^i \in \mathcal P(\overline\Omega)$ and that the goal of each agent of population $i$ is to reach a given target set $\Xi_i$ satisfying \ref{Hypo3-Gamma}. The cost optimized by an agent of population $i$ is of the form \eqref{eq:cost}, with $\Psi(\tau(\gamma))$ replaced by $\Psi(\tau_i(\gamma))$, where $\tau_i(\gamma)$ is defined as in \eqref{eq:defi3-tau} with $\Xi$ replaced by $\Xi_i$, and the integral of $h$ in \eqref{eq:cost} is replaced by the sum of two integrals, one with respect to $Q_1$, in which $\tau(\gamma) \wedge \tau(\widetilde\gamma)$ is replaced by $\tau_i(\gamma) \wedge \tau_1(\widetilde\gamma)$, and another with respect to $Q_2$, in which $\tau(\gamma) \wedge \tau(\widetilde\gamma)$ is replaced by $\tau_i(\gamma) \wedge \tau_2(\widetilde\gamma)$. The analysis carried out in Sections~\ref{sec:abstract} and \ref{sec:application} can be easily adapted to this two-population setting.

In our illustration, we consider that agents move in the square $\Omega = (0, 1) \times (0, 1) \subset \mathbb R^2$. Agents of population $1$ are initially distributed according to the uniform measure $m_0^1$ in the left side of the square, $\{0\} \times [0, 1]$, and their goal is to reach the target set $\Xi_1 = \{1\} \times [0, 1]$. Agents of population $2$ are initially distributed according to the uniform measure $m_0^2$ in the right side of the square, $\{1\} \times [0, 1]$, and their goal is to reach the target set $\Xi_2 = \{0\} \times [0, 1]$. For the illustration, we have chosen the functions $\ell$ and $\Psi$ as
\[
\ell(t, x, p) = \frac{1}{2} \abs{p}^2, \qquad \Psi(t) = t.
\]
As for the function $h$, our aim is to select a function such that, when computed along two trajectories $\gamma$ and $\widetilde\gamma$, gives
\begin{align*}
h(t, \gamma(t), \widetilde\gamma(t), \dot\gamma(t), \dot{\widetilde\gamma}(t)) 
= 8 \exp\left(-\frac{\abs{\gamma(t) - \widetilde\gamma(t)}^2}{2 \sigma^2}\right) \max\left[0, -\frac{\diff}{\diff t} \abs{\gamma(t) - \widetilde\gamma(t)}\right],
\end{align*}
where $\sigma > 0$, since the above expression penalizes two agents that are close to each other, but only when the distance $\abs{\gamma(t) - \widetilde\gamma(t)}$ between them is decreasing. As
\begin{equation*}
\frac{\diff}{\diff t} \abs{\gamma(t) - \widetilde\gamma(t)} = (\dot\gamma(t) - \dot{\widetilde\gamma}(t)) \cdot \frac{\gamma(t) - \widetilde\gamma(t)}{\abs{\gamma(t) - \widetilde\gamma(t)}},
\end{equation*}
we then wish to choose
\begin{equation*}
h(t, x, \widetilde x, p, \widetilde p) = 8 \exp\left(-\frac{\abs{x - \widetilde x}^2}{2 \sigma^2}\right) \max\left[0, -(p - \widetilde p) \cdot \frac{x - \widetilde x}{\abs{x - \widetilde x}}\right].
\end{equation*}
However, the above expression is not defined for $x = \widetilde x$, and it is not possible to provide a continuous extension of it. We then choose a parameter $\delta > 0$ and replace the term $-(p - \widetilde p) \cdot \frac{x - \widetilde x}{\abs{x - \widetilde x}}$ by a convex combination of itself and the continuous expression $\abs{p - \widetilde p}$, with weights that are respectively $1$ and $0$ when $\abs{x - \widetilde x} \geq \delta$, and that have an affine dependence on $\abs{x - \widetilde x}$ when this term is less than $\delta$. Hence, the final expression of the function $h$ chosen for our simulation is
\begin{equation}
\label{eq:simul-parameters}
\begin{aligned}
h(t, x, \widetilde x, p, \widetilde p) = 8 \exp\left(-\frac{\abs{x - \widetilde x}^2}{2 \sigma^2}\right) \max\Biggl[0, & -\min\left(1, \frac{\abs{x - \widetilde x}}{\delta}\right) (p - \widetilde p) \cdot \frac{x - \widetilde x}{\abs{x - \widetilde x}} \\
& \qquad + \max\left(0, 1 - \frac{\abs{x - \widetilde x}}{\delta}\right) \abs{p - \widetilde p}\Biggr],
\end{aligned}
\end{equation}
and we selected $\sigma = \frac{1}{4}$ and $\delta = \frac{1}{5}$.

\begin{figure}[ht]
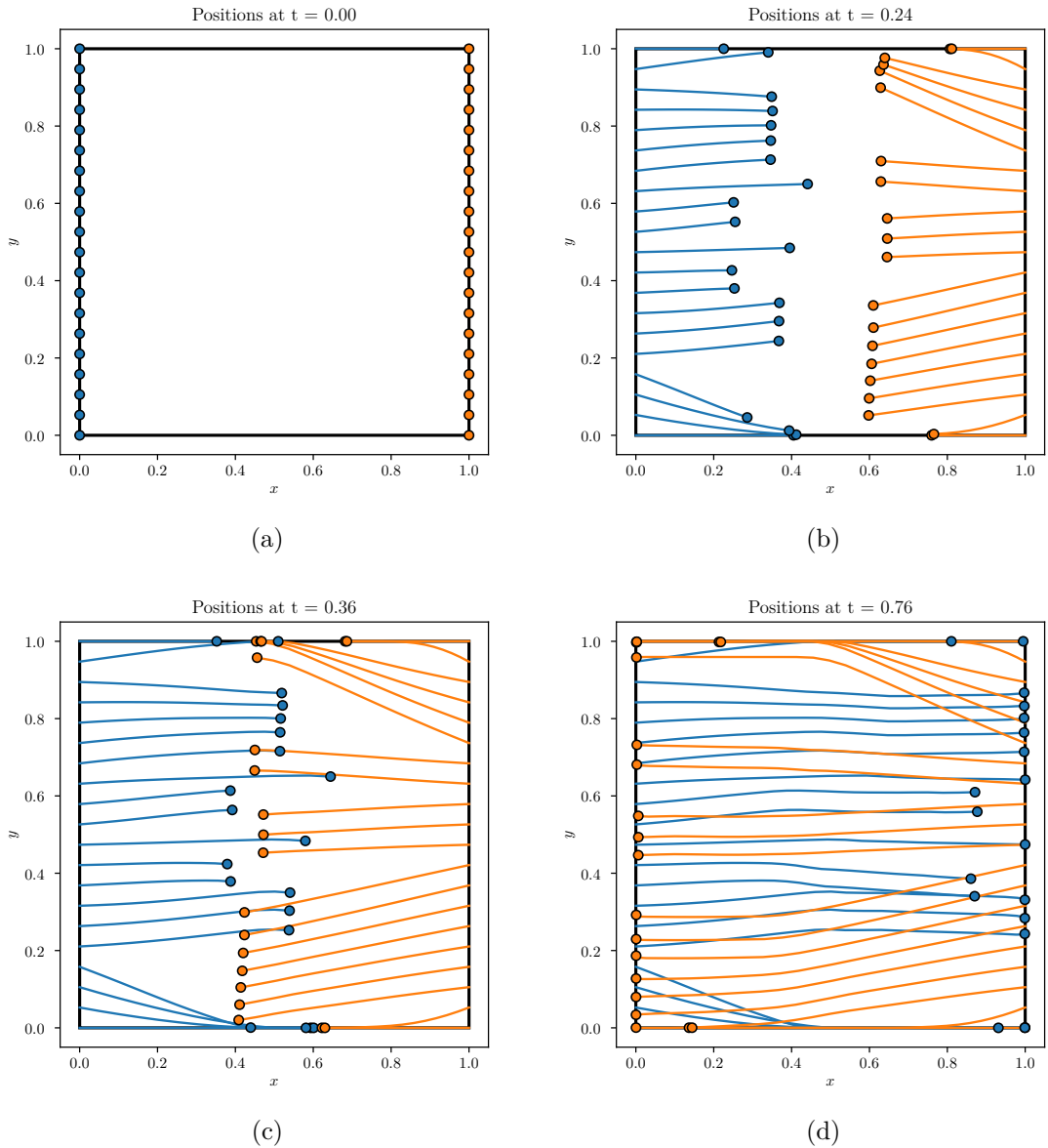

\centering
\begin{tabular}{@{} >{\centering} m{0.5\textwidth} @{} >{\centering} m{0.5\textwidth} @{}}
\resizebox{0.5\textwidth}{!}{\input{frame_0.pgf}} & \resizebox{0.5\textwidth}{!}{\input{frame_8.pgf}} \tabularnewline
(a) & (b) \tabularnewline
\resizebox{0.5\textwidth}{!}{\input{frame_12.pgf}} & \resizebox{0.5\textwidth}{!}{\input{frame_25.pgf}} \tabularnewline
(c) & (d) \tabularnewline
\end{tabular}
\caption[Trajectories of the agents of the game described in Section~\ref{sec:illustration}]{Trajectories of the agents of the game described in Section~\ref{sec:illustration} at times (a) $t = 0$, (b) $t = 0.24$, (c) $t = 0.36$, and (d) $t = 0.76$. Each population is represented by a different color, trajectories are represented by solid lines, and circles represent the current positions of the agents.}
\label{fig:simul_mfg}
\end{figure}

We have simulated the $N$-player game with $20$ players in each population, and the corresponding results are provided in Figure~\ref{fig:simul_mfg}. For the simulation, we have replaced the space $\Gamma$ by $\mathcal C([0, T], \overline\Omega)$ with $T = 3$ (this choice was validated a posteriori by noticing that all agents in the simulation reach their desired target set much before the final time $T$). The time interval $[0, T]$ was discretized in $N_t = 100$ equally spaced points and trajectories where thus represented as elements of $\overline\Omega^{N_t}$. The initial distribution $m_0^1$ was replaced by a discrete distribution concentrated on $20$ equally spaced points in the segment $\{0\} \times [0, 1]$, with a similar discretization for $m_0^2$. For a given trajectory $\gamma$ of population $i$ represented as a sequence of $N_t$ points in $\overline\Omega$, its velocity was computed as the sequence of $N_t - 1$ points obtained by dividing the difference between two successive positions in the trajectory by the time step $\Delta t = \frac{T}{N_t - 1}$, while its exit time $\tau(\gamma)$ was approximated by $\int_0^T \chi(\gamma(t)) \diff t$, where $\chi$ is a smooth function approximating the indicator function $\mathbbm 1_{\overline\Omega \setminus \Xi_i}$. The nonsmooth functions $\max(0, \cdot)$ and $\min(1, \cdot)$ appearing in \eqref{eq:simul-parameters} were replaced by smooth approximations, while the integrals from \eqref{eq:cost} were approximated using the rectangle method. Algorithm~\ref{algo:CoordinateDescent} was initialized with trajectories going in a straight line with constant velocity from their initial position to the point in the target set with the same ordinate as the initial position, while, for the minimization step in Algorithm~\ref{algo:CoordinateDescent}, we used the function \texttt{minimize} from Python's \texttt{scipy.optimize} toolbox \cite{Virtanen2020SciPy}, alternating between Nelder--Mead and L-BFGS-B methods for optimization.

We observe, in Figure~\ref{fig:simul_mfg}, a behavior close to the expected qualitative behavior for a crowd motion model. Agents reach the target set while minimizing the cost \eqref{eq:cost}, and one can see the effect of the interaction term represented by $h$ by noticing that agents seem to anticipate the arrival of agents in the opposite direction, leaving them some space to pass in order to avoid a large congestion term due to $h$.

\end{document}